\documentclass[preprint,12pt]{elsarticle}
\usepackage{epsfig}
\usepackage{amssymb}
\usepackage{amsmath}
\usepackage{dsfont}
\usepackage{pb-diagram}
\usepackage{youngtab}
\usepackage{amsthm}
\newtheorem{thm}{Theorem}[section]
\newtheorem{lem}[thm]{Lemma}

\newtheorem{cor}[thm]{Corollary}
\newdefinition{defn}[thm]{Definition}
\newdefinition{exmp}[thm]{Example}

\newdefinition{rem}[thm]{Remark}
\numberwithin{equation}{thm}
\allowdisplaybreaks

\begin{document}

\begin{frontmatter}
\title{Quantized mixed tensor space and Schur--Weyl duality}

\author[s]{R.~Dipper }
\ead{rdipper@mathematik.uni-stuttgart.de}

\author[c]{S.~Doty }
\ead{doty@math.luc.edu}

\author[s]{F.~Stoll }
\ead{stoll@mathematik.uni-stuttgart.de}

\address[s]{ Institut f\"ur Algebra und Zahlentheorie,
  Universit\"at Stuttgart,
  Pfaffenwaldring 57, 70569 Stuttgart,
  Germany}
\address[c]{  Department of Mathematics and Statistics,
  Loyola University Chicago,
  1032 W.~Sheridan Road, 
  Chicago,IL 60660 USA}

\begin{abstract}
  Let $R$ be a commutative ring with one and $q$ an invertible element
  of $R$. The (specialized) quantum group ${\mathbf
    U}=U_q(\mathfrak{gl}_n)$ over $R$ of the general linear group acts
  on mixed tensor space $V^{\otimes r}\otimes {V^*}^{\otimes s}$ where
  $V$ denotes the natural $\mathbf U$-module $R^n$, $r,s$ are
  nonnegative integers and $V^*$ is the dual $\mathbf U$-module to
  $V$.  The image of $\mathbf U$ in $\mathrm{End}_R(V^{\otimes
    r}\otimes {V^*}^{\otimes s})$ is called the rational $q$-Schur
  algebra $S_{q}(n;r,s)$. We construct a bideterminant basis of
  $S_{q}(n;r,s)$. There is an action of a $q$-deformation
  $\mathfrak{B}_{r,s}^n(q)$ of the walled Brauer algebra on mixed
  tensor space centralizing the action of $\mathbf U$. We show that
  $\mathrm{End}_{\mathfrak{B}_{r,s}^n(q)}(V^{\otimes r}\otimes
  {V^*}^{\otimes s})=S_{q}(n;r,s)$. By \cite{dipperdotystoll} the
  image of $\mathfrak{B}_{r,s}^n(q)$ in $\mathrm{End}_R(V^{\otimes
    r}\otimes {V^*}^{\otimes s})$ is $\mathrm{End}_{\mathbf
    U}(V^{\otimes r}\otimes {V^*}^{\otimes s})$. Thus mixed tensor
  space as $\mathbf U$-$\mathfrak{B}_{r,s}^n(q)$-bimodule satisfies
  Schur--Weyl duality.
\end{abstract}


\end{frontmatter}
\input epsf

\section*{Introduction}
Schur--Weyl duality plays an important role in representation theory
since it relates the representations of the general linear group with
the representations of the symmetric group. The classical Schur--Weyl
duality due to Schur (\cite{schur}) states that the actions of the
general linear group $G=\mathrm{GL}_n(\mathbb{C})$ and the symmetric
group $\mathfrak{S}_m$ on tensor space $V^{\otimes m}$ with
$V=\mathbb{C}^n$ and satisfy the bicentralizer property, that is
$\mathrm{End}_{\mathfrak{S}_m}(V^{\otimes m})$ is generated by the
action of $G$ and correspondingly, $\mathrm{End}_{G}(V^{\otimes m})$
is generated by the action of $\mathfrak{S}_m$.  This duality has been
generalized to subgroups of $G$ (e.~g.~orthogonal, symplectic groups,
Levi subgroups) and corresponding algebras related with the group
algebra of the symmetric group (e.~g.~Brauer algebras, Ariki-Koike
algebras), as well as deformations of these algebras. In general, the
phrase 'Schur--Weyl duality' has come to indicate such a bicentralizer
property for two algebras acting on some module.

One such generalization is the mixed tensor space $V^{\otimes
  r}\otimes {V^*}^{\otimes s}$ where $V$ is the natural and $V^*$ its
dual $\mathbb{C}G$-module.  The centralizer algebra is known to be the
walled Brauer algebra $\mathfrak{B}_{r,s}^n$ and it was shown in
\cite{bchlls} by Benkart, Chakrabarti, Halverson, Leduc, Lee and
Stroomer, that mixed tensor space under the action of $\mathbb{C}G$
and $\mathfrak{B}_{r,s}^n$ satisfies Schur--Weyl duality; see also
\cite{Koike} and \cite{Turaev}.  Kosuda and Murakami
\cite{kosudamurakami} introduced a one parameter deformation
$\mathfrak{B}_{r,s}^n(q)$ of the walled Brauer algebra and proved
Schur--Weyl duality in the generic case (i.~e.~over $\mathbb{C}(q)$),
where $\mathbb{C}G$ is replaced by the generic quantum group
$U_{\mathbb{C}(q)}(\mathfrak{gl}_n)$.

In this paper, we generalize the results of
\cite{bchlls,kosudamurakami} for a very general setting.  Let $R$ be a
commutative ring with $1$ and $q\in R$ be invertible.  Let $\mathbf U$
be (a specialized version of) the quantum group over $R$, which replaces
the general linear group in the quantized case.  Let
$\mathfrak{B}_{r,s}^n(q)$ be the $q$-deformation of the walled Brauer
algebra defined by Leduc \cite{leduc}.  Here we use a specialized
version of the multi-parameter version of \cite{leduc} which acts on
mixed tensor space $V^{\otimes r}\otimes {V^*}^{\otimes s}$ where
$V=R^n$ is the natural $\mathbf U$-module.

 In \cite{dipperdotystoll} one side of Schur--Weyl duality was shown in
 this situation, namely that the image of $\mathfrak{B}_{r,s}^n(q)$ in 
$\mathrm{End}_R(V^{\otimes r}\otimes {V^*}^{\otimes s})$ is the
centralizing algebra of the action of $\mathbf U$ on mixed tensor
space. 

In this paper, which is a revised version of a preprint which has
circulated since 2008, the other side of Schur--Weyl duality will be
proven, namely that the image of $\mathbf U$ in
$\mathrm{End}_R(V^{\otimes r}\otimes {V^*}^{\otimes s})$ is the
endomorphism algebra of mixed tensor space under the action of
$\mathfrak{B}_{r,s}^n(q)$. We call this image the \emph{rational
  $q$-Schur algebra} and denote it by $S_q(n;r,s)$. It is a
$q$-analogue of the rational Schur algebra introduced and studied in
\cite{dipperdoty}. In case $q=1$ we obtain a similar statement (which
is also new) for the rational Schur algebra with respect to the
hyperalgebra over $R$ of $\mathfrak{gl}_n$. In the meantime, Tange
\cite{tange} showed this result in the special case $q=1$ by different
methods. One may also wish to consult the recent paper
\cite{BrundanStroppel} by Brundan and Stroppel, which enlarges the
landscape on walled Brauer algebras considerably.

For technical reasons it will be useful to turn things around and
instead define $S_q(n;r,s)$ to be
$\mathrm{End}_{\mathfrak{B}_{r,s}^n(q)}(V^{\otimes r}\otimes
{V^*}^{\otimes s})$. Since we show at the end that this coincides with
the image of $\mathbf U$ in $\mathrm{End}_R(V^{\otimes r}\otimes
{V^*}^{\otimes s})$ there is no harm in this abuse of notation. In the
course of our proof we will show that
$\mathrm{End}_{\mathfrak{B}_{r,s}^n(q)}(V^{\otimes r}\otimes
{V^*}^{\otimes s})=S_q(n;r,s)$ is free as $R$-module of rank
independent of the choice of $R$ and $q$. We shall accomplish this by
constructing an $R$-basis of $S_q(n;r,s)$ which is dual to a certain
bideterminant basis of the dual coalgebra $A_q(n;r,s)$ of
$S_q(n;r,s)$.

As a guide for the reader we briefly outline the main ideas behind the
proof. There is a natural embedding of mixed tensor space $V^{\otimes
  r}\otimes {V^*}^{\otimes s}$ into ordinary tensor space $V^{\otimes
  r+(n-1)s}$. This embedding $\kappa$ is not $\mathbf{U}$-linear, but is
$\mathbf U'$-linear, where $\mathbf U'$ is the subalgebra of $\mathbf
U$ corresponding to the special linear Lie algebra. We will see that
replacing $\mathbf U$ by $\mathbf U'$ is not significant. For
$u\in\mathbf U'$ the restriction of the action of $u$ on $V^{\otimes
  r+(n-1)s}$ to $V^{\otimes r}\otimes {V^*}^{\otimes s} \leq
V^{\otimes r+(n-1)s}$ commutes with the action of
$\mathfrak{B}_{r,s}^n(q)$ on $V^{\otimes r}\otimes {V^*}^{\otimes s} $,
hence lies in $S_q(n;r,s)$.  Thus $\kappa$ induces an algebra
homomorphism $\pi$ from the ordinary $q$-Schur algebra $S_q(n,
r+(n-1)s)$ which is the image of $\mathbf U'$ in
$\mathrm{End}_R(V^{\otimes r+(n-1)s}$ into $S_q(n;r,s)$.  This
homomorphism was motivated by a similar homomorphism in
\cite{dipperdoty}.

Let $\rho_{\mathrm{ord}}:\mathbf U'\to S_q(n,r+(n-1)s)$ 
be the representation  of $\mathbf U'$ on $ V^{\otimes
  r+(n-1)s}$ and  
$\rho_{\mathrm{mxd}}: \mathbf U'\to S_q(n;r,s)$  the
representation of $\mathbf U'$ on mixed tensor space. 
Then $\rho_{\mathrm{mxd}}=\pi\circ \rho_{\mathrm{ord}}$ by
construction. By classical quantized Schur--Weyl duality
$\rho_{\mathrm{ord}}$ is surjective, hence  $\rho_{\mathrm{mxd}}$ is
surjective (i.e.~$\rho_{\mathrm{mxd}}(\mathbf U')= S_q(n;r,s)$)
if $\pi$ is surjective. We show that $\pi$ possesses an $R$-linear
right inverse, thus proving the surjectivity of $\pi$.

At this point, we switch over to  coefficient spaces: It is well
known that the dual coalgebra $A_q(n,r+(n-1)s)=S_q(n,r+(n-1)s)^*$ is
the coefficient space of $ \mathbf U'$ acting on ordinary tensor space
$ V^{\otimes  r+(n-1)s}$. There is no problem here with dualisation,
since the classical $q$-Schur algebra $ S_q(n,r+(n-1)s)$ is known to
be free as $R$-module of fixed rank independent of the choice of $R$
and $q$. Moreover $A_q(n,r+(n-1)s)$ possesses a bideterminant basis
(\cite{huangzhang}).  The endomorphism algebra $S_q(n;r,s)=
\mathrm{End}_{\mathfrak{B}_{r,s}^n(q)}(V^{\otimes  r}\otimes
{V^*}^{\otimes s})$  may be described by a system of linear equations
in the endomorphism algebra 
$\mathrm{End}_R(V^{\otimes  r}\otimes{V^*}^{\otimes s})$, which is 
free as $R$-module. Using these equations, we apply a general argument
(Lemma~\ref{lem:construction}) to construct a factor coalgebra
$A_q(n;r,s)$ of the $R$-coalgebra $\mathrm{End}_R(V^{\otimes
  r}\otimes{V^*}^{\otimes s})$  such that $A_q(n;r,s)^*$ is isomorphic
to the $R$-algebra $S_q(n;r,s)$. In Section~\ref{sec:iota} we exhibit
a map $\iota:A_q(n;r,s)\to A_q(n,r+(n-1)s)$ and show explicitly that
$\iota^*=\pi: S_q(n, r+(n-1)s)\to S_q(n;r,s)$. In Section~\ref{section:basis}
we show that $A_q(n;r,s)$ and hence  $S_q(n;r,s)$ is free as
$R$-module by constructing a (rational) bideterminant basis. From this
it is not hard to find an ($R$-linear) left inverse of the map $\iota$
whose dual map is then the required right inverse of $\iota^*=\pi$,
proving that  $S_q(n;r,s)$ is the image of $\mathbf U'$ (and hence
$\mathbf U$) acting on mixed tensor space.

\section{Preliminaries}\label{section:hopf}
Let $n$ be a given positive integer.  In this section, we introduce
the quantized enveloping algebra of the general linear Lie algebra
$\mathfrak{gl}_n$ over a commutative ring $R$ with parameter $q$ and
summarize some well known results; see for example
\cite{hongkang,jantzen,lusztig}.  We will start by recalling the
definition of the quantized enveloping algebra over $\mathbb{Q}(q)$
where $q$ is an indeterminate.

Let $P^\vee$ be the free $\mathbb{Z}$-module with basis $h_1, \ldots,
h_n$ and let $\varepsilon_1, \ldots, \varepsilon_n \in {P^\vee}^*$ be
the corresponding dual basis: $\varepsilon_i$ is given by
$\varepsilon_i(h_j):=\delta_{i,j}$ for $j=1, \ldots, n$, where
$\delta$ is the usual Kronecker symbol.  For $i=1,\ldots,n-1$ let
$\alpha_i\in {P^\vee}^*$ be defined by $\alpha_i := \varepsilon_i -
\varepsilon_{i+1}$. 
\begin{defn}
  The quantum general linear algebra $U_q(\mathfrak{gl}_n)$ is the
  associative $\mathbb{Q}(q)$-algebra with $1$ 
  generated by the elements
  $e_i, f_i$ $(i=1, \ldots, n-1)$ and $q^{h}$ $(h\in P^\vee)$
  with the defining relations
  \begin{eqnarray*}
    &&q^0=1,\quad q^hq^{h'}=q^{h+h'}\\
    &&q^he_iq^{-h}=q^{\alpha_i(h)}e_i,\quad 
    q^hf_iq^{-h}=q^{-\alpha_i(h)}f_i,\\
    &&e_if_j-f_je_i=\delta_{i,j}\frac{K_i-K_i^{-1}}{q-q^{-1}},\quad
    \mbox{ where }K_i := q^{h_{i}-h_{i+1}},\\
    &&e_i^2e_j-(q+q^{-1})e_ie_je_i+e_je_i^2=0\quad\mbox{ for }|i-j|=1,\\
    &&f_i^2f_j-(q+q^{-1})f_if_jf_i+f_jf_i^2=0\quad\mbox{ for }|i-j|=1,\\
    &&e_ie_j=e_je_i,\quad f_if_j=f_jf_i\quad\mbox{ for }|i-j|>1.
  \end{eqnarray*}
\end{defn}
We note that the subalgebra generated by the $K_i, e_i, f_i$ ($i = 1,
\dots, n-1$) is isomorphic with $U_q(\mathfrak{sl}_n)$.
$U_q(\mathfrak{gl}_n)$ is a Hopf algebra with comultiplication
$\Delta$, counit $\varepsilon$ the unique algebra homomorphisms, and
antipode $S$ the unique invertible anti-homomorphism of algebras,
defined on generators by
\begin{eqnarray*}
  &&\Delta(q^h)=q^h\otimes q^h,\\
  &&\Delta(e_i)=e_i\otimes K_i^{-1}+1\otimes e_i,\quad
  \Delta(f_i)=f_i\otimes 1+K_i\otimes f_i,\\
  &&\varepsilon(q^h)=1,\quad \varepsilon(e_i)=\varepsilon(f_i)=0,\\
  && S(q^h)=q^{-h},\quad S(e_i)=-e_iK_i,\quad S(f_i)=-K_i^{-1}f_i .
\end{eqnarray*}
Let $V_{\mathbb{Q}(q)}$ be a free $\mathbb{Q}(q)$-vector space with
basis $\{v_1,\ldots,v_n\}$. We make $V_{\mathbb{Q}(q)}$ into a
$U_q(\mathfrak{gl}_n)$-module via
\begin{eqnarray*} 
  q^h v_j&=& 
  q^{\varepsilon_j(h)}v_j \mbox{ for }h\in P^\vee,\,j=1,\ldots,n\\
  e_iv_j&=&\left\{
    \begin{array}{ll}
      v_i&\mbox{ if }j=i+1\\
      0&\mbox{ otherwise}
    \end{array}
  \right. \hspace{3cm}
  f_iv_j=\left\{
    \begin{array}{ll}
      v_{i+1}&\mbox{ if }j=i\\
      0&\mbox{ otherwise.}
    \end{array}
  \right.
\end{eqnarray*}
We call $V_{\mathbb{Q}(q)}$ the \emph{vector representation} of
$U_q(\mathfrak{gl}_n)$.  This is also a $U_q(\mathfrak{sl}_n)$-module,
by restriction of the action.

Let $[l]_q$ (in $\mathbb{Z}[q,q^{-1}]$ resp.~in $R$) be defined by
$[l]_q:=\sum_{i=0}^{l-1}q^{2i-l+1}$, $[l]_q!:=[l]_q[l-1]_q\ldots[1]_q$
and let $e_i^{(l)}:=\frac{e_i^l}{[l]_q!}$,
$f_i^{(l)}:=\frac{f_i^l}{[l]_q!}$.  Let $\mathbf
U_{\mathbb{Z}[q,q^{-1}]}$ (resp.,
$\mathbf{U}'_{\mathbb{Z}[q,q^{-1}]}$) be the
$\mathbb{Z}[q,q^{-1}]$-subalgebra of $U_q(\mathfrak{gl}_n)$ generated
by the $q^h$ (resp., $K_i$) and the divided powers $e_i^{(l)}$ and
$f_i^{(l)}$ for $l\geq 0$. $\mathbf U_{\mathbb{Z}[q,q^{-1}]}$ is a
Hopf algebra and we have
\begin{eqnarray*} 
  \Delta(e_i^{(l)})&=&
  \sum_{k=0}^l q^{k(l-k)}e_i^{(l-k)}\otimes K_i^{k-l} e_i^{(k)}\\
  \Delta(f_i^{(l)})&=&
  \sum_{k=0}^l q^{-k(l-k)}f_i^{(l-k)}K_i^k\otimes f_i^{(k)}\\
  S(e_i^{(l)})&=& (-1)^lq^{l(l-1)}e_i^{(l)}K_i^l\\
  S(f_i^{(l)})&=& (-1)^lq^{-l(l-1)}K_i^{-l}f_i^{(l)}\\
  \varepsilon(e_i^{(l)})&=&\varepsilon(f_i^{(l)})=0.
\end{eqnarray*}
Furthermore, the $\mathbb{Z}[q,q^{-1}]$-lattice
$V_{\mathbb{Z}[q,q^{-1}]}$ in $V_{\mathbb{Q}(q)}$ generated by the
$v_i$ is invariant under the action of $\mathbf
U_{\mathbb{Z}[q,q^{-1}]}$ and of $\mathbf{U}'_{\mathbb{Z}[q,q^{-1}]}$.
Now, make the transition from $\mathbb{Z}[q,q^{-1}]$ to an arbitrary
commutative ring $R$ with $1$: Let $q\in R$ be invertible and consider
$R$ as a $\mathbb{Z}[q,q^{-1}]$-module via specializing $q\in
\mathbb{Z}[q,q^{-1}]\mapsto q\in R$. Then, let $\mathbf
U_R:=R\otimes_{\mathbb{Z}[q,q^{-1}]}\mathbf U_{\mathbb{Z}[q,q^{-1}]}$
and
$\mathbf{U}'_R:=R\otimes_{\mathbb{Z}[q,q^{-1}]}\mathbf{U}'_{\mathbb{Z}[q,q^{-1}]}$.
$\mathbf U_R$ inherits a Hopf algebra structure from $\mathbf
U_{\mathbb{Z}[q,q^{-1}]}$ and
$V_R:=R\otimes_{\mathbb{Z}[q,q^{-1}]}V_{\mathbb{Z}[q,q^{-1}]}$ is a
$\mathbf U_R$-module and by restriction also a $\mathbf{U}'_R$-module.

If no ambiguity arises, we will henceforth omit the index $R$ and
write $\mathbf{U}$, $\mathbf{U}'$ instead of $\mathbf{U}_R$,
$\mathbf{U}'_R$ and $V$ instead of $V_R$. Furthermore, we will write
$e_i^{(l)}$ as shorthand for $1\otimes e_i^{(l)}\in \mathbf{U}_R$,
similarly for the $f_i^{(l)}$, $K_i$ short for $1\otimes K_i$, and
$q^h$ short for $1\otimes q^h$.

Suppose $W, W_1$ and $W_2$ are $\mathbf U$-modules, then one can
define $\mathbf U$-module structures on $W_1\otimes W_2=W_1\otimes_R
W_2$ and $W^*=\mathrm{Hom}_R(W,R)$ using the comultiplication and the
antipode  by setting $x(w_1\otimes
w_2)=\Delta(x)(w_1\otimes w_2)$ and $(xf)(w)=f(S(x)w)$. 

\begin{defn}
  Let $r,s$ be nonnegative integers. 
  The $\mathbf U$-module $V^{\otimes r}\otimes {V^*}^{\otimes s}$ is
  \emph{called mixed tensor space}.
\end{defn}
Let $I(n,r)$ be the set of $r$-tuples with entries in $\{1,\ldots,n\}$
and let $I(n,s)$ be defined similarly.
The elements of $I(n,r)$ (and $I(n,s)$) are called \emph{multi indices}.
Note that the symmetric groups
$\mathfrak{S}_r$ and $\mathfrak{S}_s$ act on  $I(n,r)$ and  $I(n,s)$
respectively from the right
by place permutation, that is if $s_j$ is a Coxeter generator and
$\mathbf i=(i_1,i_2,\ldots )$ is a multi index, then let $\mathbf
i.s_j=(i_1,\ldots,i_{j-1},i_{j+1},i_j,i_{j+2},\ldots)$. 
Then a basis of the mixed
tensor space $V^{\otimes r}\otimes {V^*}^{\otimes s}$ can be indexed
by $I(n,r)\times I(n,s)$. For $\mathbf i=(i_1,\ldots,i_r)\in I(n,r)$ and
$\mathbf j=(j_1,\ldots,j_s)\in I(n,s)$ let
\[
v_{\mathbf i|\mathbf
  j}=v_{i_1}\otimes \ldots\otimes v_{i_r}\otimes v_{j_1}^*\otimes
\ldots\otimes v_{j_s}^*\in V^{\otimes r}\otimes {V^*}^{\otimes s}
\]
where $\{v_1^*,\ldots,v_n^*\}$ is the basis of $V^*$ dual to
$\{v_1,\ldots,v_n\}$. 
Then $\{v_{\mathbf i|\mathbf j}\mid \mathbf i\in I(n,r),\mathbf j\in
I(n,s)\}$ is a basis of $
V^{\otimes r}\otimes {V^*}^{\otimes s}$. 

We have another algebra acting on $V^{\otimes r}\otimes {V^*}^{\otimes
  s}$, namely the quantized walled Brauer algebra
$\mathfrak{B}_{r,s}^n(q)$ introduced in \cite{dipperdotystoll}. This
algebra is defined as a diagram algebra, in terms of Kauffman's
tangles. A presentation by generators and relations can be found in
\cite{dipperdotystoll}. Note that this
algebra and its action coincides with Leduc's algebra (\cite{leduc},
see the remarks in \cite{dipperdotystoll}).

Here, all we need is the action of generators given in the following
diagrams.  $\mathfrak{B}_{r,s}^n(q)$ is generated by the elements
\[
E=\raisebox{-.1cm}{\epsfbox{generators.3}},\quad
S_i=\raisebox{-.1cm}{\epsfbox{generators.1}},\quad
\hat{S}_j=\raisebox{-.1cm}{\epsfbox{generators.2}}
\]
where the non-propagating edges in $E$ connect vertices in columns
$r$, $r+1$ while the crossings in $S_i$ and $\hat{S}_j$ connect
vertices in columns $i$, $i+1$ and columns $r+j$, $r+j+1$
respectively. If $v_{\mathbf i|\mathbf j}=v\otimes v_{i_r}\otimes
v_{j_1}^*\otimes v'$, then the action of the generators on $V^{\otimes
  r}\otimes {V^*}^{\otimes s}$ is given by
\begin{eqnarray*}
  v_{\mathbf i|\mathbf j}E&=&\delta_{i_r,j_1}\sum_{s=1}^nq^{2i_r-n-1}
  v\otimes v_s\otimes v_s^*\otimes v'\\
  v_{\mathbf i|\mathbf j}S_i&=& \left\{
    \begin{array}{ll}
      q^{-1} v_{\mathbf i|\mathbf j}&\mbox{ if }i_i=i_{i+1}\\
      v_{\mathbf i.s_i|\mathbf j}&\mbox{ if }i_i<i_{i+1}\\
      v_{\mathbf i.s_i|\mathbf j}+ (q^{-1}-q)
      v_{\mathbf i|\mathbf j}&\mbox{ if }i_i> i_{i+1}
    \end{array}
  \right.\\
  v_{\mathbf i|\mathbf j}\hat{S}_j&=& \left\{
    \begin{array}{ll}
      q^{-1} v_{\mathbf i|\mathbf j}&\mbox{ if }j_j=j_{j+1}\\
      v_{\mathbf i|\mathbf j.s_j}&\mbox{ if }j_j>j_{j+1}\\
      v_{\mathbf i|\mathbf j.s_j}+ (q^{-1}-q)
      v_{\mathbf i|\mathbf j}&\mbox{ if }j_j<j_{j+1}.
    \end{array}
  \right.
\end{eqnarray*}

The action of $\mathfrak{B}_{r,s}^n(q)$ 
on $V^{\otimes r}\otimes {V^*}^{\otimes s}$ commutes with the action
of $\mathbf U$. 
\begin{thm}[\cite{dipperdotystoll}]\label{thm:Schur--Weyl_I}
  Let $\sigma: \mathfrak{B}_{r,s}^n(q)\to
  \mathrm{End}_{\mathbf U}(V^{\otimes r}\otimes
  {V^*}^{\otimes s}) $ be the representation of the quantized walled
  Brauer algebra on the mixed tensor space. 
  Then $\sigma$ is surjective, that is
  \[
  \mathrm{End}_{\mathbf U}(V^{\otimes r}\otimes
  {V^*}^{\otimes s})
  \cong 
  \mathfrak{B}_{r,s}^n(q)/_{\mathrm{ann}_{\mathfrak{B}_{r,s}^n(q)}(V^{\otimes r}\otimes
    {V^*}^{\otimes s})}.
  \] 
\end{thm}

The main result of this paper is the other half of the preceding
theorem:
\begin{thm}\label{thm:Schur--Weyl_II}
  Let $\rho_{\mathrm{mxd}}: \mathbf U\to\mathrm{End}_{ \mathfrak{B}_{r,s}^n(q) } 
  (V^{\otimes r}\otimes
  {V^*}^{\otimes s}) $ be the representation of the quantum group. 
  Then $\rho_{\mathrm{mxd}}$ is surjective, that is
  \[
  \mathrm{End}_{ \mathfrak{B}_{r,s}^n(q) }(V^{\otimes r}\otimes
  {V^*}^{\otimes s})
  \cong \mathbf U
  /_{\mathrm{ann}_{\mathbf U}(V^{\otimes r}\otimes
    {V^*}^{\otimes s})}.
  \] 
\end{thm}

Theorems~\ref{thm:Schur--Weyl_I} and \ref{thm:Schur--Weyl_II} together
state that the mixed tensor space is a $(\mathbf
U, \mathfrak{B}_{r,s}^n(q))$-bimodule with the double centralizer
property. In the literature, this is also called \emph{Schur--Weyl
  Duality}.  Theorem~\ref{thm:Schur--Weyl_II} will be proved at the end
of this paper.

For $s=0$, this is well known. $\mathfrak{B}_{m,0}^n(q)$ is the Hecke
algebra $\mathcal{H}_m$, and $V^{\otimes m}$ is the (ordinary) tensor
space. 

\begin{defn}
  If $m$ is a positive integer, let $\mathcal{H}_m$ be the associative
  $R$-algebra with one generated by elements $T_1,\ldots,T_{m-1}$ with
  respect to the relations
  \begin{align*}
    &(T_i+q)(T_i-q^{-1})=0 \mbox{ for }i=1,\ldots,m-1\\
    &T_iT_{i+1} T_i=T_{i+1} T_iT_{i+1}\mbox{ for }i=1,\ldots,m-2\\
    &T_iT_j=T_jT_i\mbox{ for } |i-j|\geq 2.
  \end{align*}
\end{defn}

  If $w\in\mathfrak{S}_m$ is an element of the symmetric group on $m$
  letters, and $w=s_{i_1}s_{i_2}\ldots s_{i_l}$ is a reduced
  expression as a product of Coxeter generators, let
  $T_w=T_{i_1}T_{i_2}\ldots T_{i_l}$. Then the set $\{T_w\mid
  w\in\mathfrak{S}_m\}$ is a basis of $\mathcal{H}_m$. 

  Note that $\mathcal{H}_m$ acts on $V^{\otimes m}$, since 
  $\mathcal{H}_m\cong \mathfrak{B}_{m,0}^n(q)$, the isomorphism given
  by $T_i\mapsto S_i$. 

\begin{thm}[\cite{dipperjamesschur,green}]
  Let $\rho_{\mathrm{ord}}:\mathbf{U} \to \mathrm{End}_R(V^{\otimes
    m})$ be the representation of $\mathbf U$ on $ V^{\otimes
    m}$. Then $\operatorname{im} \rho_{\mathrm{ord}} =
  \mathrm{End}_{\mathcal{H}_m}(V^{\otimes m})$.  This algebra is
  called the $q$-Schur algebra and denoted by $S_q(n,m)$.
\end{thm}

 We will refer to $V^{\otimes m}$ as ordinary tensor space.

\section{Mixed tensor space as a submodule}\label{section:submodule}

Recall that $\mathbf{U}'$ is the subalgebra of $\mathbf{U}$
corresponding to the Lie algebra $\mathfrak{sl}_n$.

\begin{thm}
  If $m$ is a nonnegative integer, let $\rho_{\mathrm{ord}}:\mathbf U\to
  \mathrm{End}_R(V^{\otimes m})$ be the representation of $\mathbf U$
  on $V^{\otimes m}$. Then 
  \[
  \rho_{\mathrm{ord}}(\mathbf U)=\rho_{\mathrm{ord}}(\mathbf U').
  \]
\end{thm}

\begin{proof}
  Define the \emph{weight} of $\mathbf i\in I(n,m)$ to be
  $\mathrm{wt}(\mathbf i)=\lambda
  =(\lambda_1,\ldots,\lambda_n)$, such that
  $\lambda_i$ is the number of entries in 
  $\mathbf i$, that are equal to
  $i$.
  If $\lambda=(\lambda_1,\ldots,\lambda_n)$ is a composition of $m$ into
  $n$ parts, i.~e.~ $\lambda_1+\ldots+\lambda_n=m$,
  let $V^{\otimes m}_\lambda$ be the $R$-submodule of
  $V^{\otimes m}$ generated by all $v_{\mathbf  i}$ with 
  $\mathrm{wt}(\mathbf i)=\lambda$. Then $V^{\otimes m}$ is the direct
  sum of all $V^{\otimes m}_\lambda$, where $\lambda$ runs through the
  set of compositions of $m$ into $n$ parts. Let $\varphi_\lambda$ be the
  projection onto  $V^{\otimes m}_\lambda$.
  \cite{green} shows, that the restriction of 
  $\rho_{\mathrm{ord}}:\mathbf U\to
  S_q(n,m)$ to any subalgebra $\mathbf U'\subseteq \mathbf U$ 
  is surjective, if
  the subalgebra $\mathbf U'$ contains
  the divided powers $e_i^{(l)},f_i^{(l)}$
  and preimages of the projections $\varphi_\lambda$. 
  
  Therefore, we define a partial order on
  the set of compositions of $m$ into $n$ parts by 
  $\lambda\preceq\mu$ if and only if 
  $(\lambda_1-\lambda_2,\lambda_2-\lambda_3,\ldots,
  \lambda_{n-1}-\lambda_n)\leq(\mu_1-\mu_2,\mu_2-\mu_3,\ldots,
  \mu_{n-1}-\mu_n)$ in the lexicographical order. 
  It suffices to show, that for each composition
  $\lambda$, there exists an element $u\in {\mathbf U}'$ such that
  $uv_{\mathbf i}=0$ whenever $\mathrm{wt}(\mathbf i)\prec \lambda$ (i.~e.~
  $\mathrm{wt}(\mathbf i)\preceq \lambda$ and 
  $\mathrm{wt}(\mathbf i)\neq \lambda$) and 
  $uv_{\mathbf i}=v_{\mathbf i}$ whenever $\mathrm{wt}(\mathbf i)=
  \lambda$. 
  In Theorem~4.5 of \cite{lusztig}, it
  is shown that certain elements 
  \[ \left[\genfrac{}{}{0pt}{}{K_i;c}{t}\right]
  :=\prod_{s=1}^t
  \frac{K_iq^{c-s+1}-K_i^{-1}q^{-c+s-1} }{q^s-q^{-s}}
  \]
  are elements of ${\mathbf U}'$ for $i=1,\ldots,n-1,c\in\mathbb{Z}$ and
  $t\in\mathbb{N}$. 
  Let 
  \[u:=\prod_{i=1}^{n-1} 
  \left[\genfrac{}{}{0pt}{}
    {K_i;m+1}{\lambda_i-\lambda_{i+1}+m+1}
  \right],
  \] 
  which is an element of ${\mathbf U}'$ since 
  $\lambda_i-\lambda_{i+1}+m+1>0$. 
  Then $u$ has the desired properties. 
\end{proof}
The next lemma is motivated by \cite[\S6.3]{dipperdoty}.
\begin{lem}\label{lem:monomorphism}
  There is a well defined ${\mathbf U}'$-monomorphism 
  $\kappa:{V^*}\to
  V^{\otimes n-1}$
  given by 
  \begin{eqnarray*}
    v_i^*&\mapsto& (-q)^i\sum_{w\in\mathfrak{S}_{n-1}}
    (-q)^{l(w)}v_{(12\ldots\hat{i}\ldots n).w}\\
    &=&(-q)^i\sum_{w\in\mathfrak{S}_{n-1}}
    (-q)^{l(w)}v_{(12\ldots\hat{i}\ldots n)}T_w
    =(-q)^i v_{(12\ldots\hat{i}\ldots n)}
    \sum_{w\in\mathfrak{S}_{n-1}}
    (-q)^{l(w)}T_w
  \end{eqnarray*}
  where $\hat{i}$ means
  leaving out $i$.
\end{lem}

\begin{proof}
  It is clear, that $\kappa$ is a monomorphism of $R$-modules. 
  By definition, $K_iv_j^*=q^{\delta_{i+1,j}-\delta_{i,j}}v_j^*$ and
  $K_i v_{(1\ldots\hat{j}\ldots n)}=
  q^{1-\delta_{i,j}}q^{\delta_{i+1,j}-1}v_{(1\ldots  \hat{j}\ldots n)}$.
  Thus $\kappa$ commutes with 
  $K_i$. 
  Now $e_iv_j^*=-\delta_{i,j}q^{-1}v_{j+1}^*$. If $j\neq i,i+1$ then 
  \begin{eqnarray*}
    e_i\kappa(v_j^*)&=&
    (-q)^je_i\sum_w(-q)^{l(w)}v_{(1\ldots i i+1\ldots \hat{j}\ldots n)}T_w\\
    &=&-(-q)^j\sum_w(-q)^{l(w)}v_{(1\ldots i i\ldots \hat{j}\ldots
      n)}T_w =0=\kappa(e_iv_j^*)
  \end{eqnarray*}
  For $j=i$ resp.~$i+1$ we get
  \begin{eqnarray*}
   e_i\kappa(v_{i+1}^*)&=&
   (-q)^{i+1}\sum_w(-q)^{l(w)}(e_iv_{(1\ldots  \widehat{i+1}\ldots n)})T_w=0\\
   e_i\kappa(v_i^*)&=&
   (-q)^i\sum_w(-q)^{l(w)}(e_iv_{(1\ldots  \hat{i}i+1\ldots n)})T_w\\
   &=&
    (-q)^{i}\sum_w(-q)^{l(w)}v_{(1\ldots  i \,\widehat{i+1}\ldots n)}T_w
    =-q^{-1}\kappa(v_{i+1}^*)
  \end{eqnarray*}
Furthermore, for $l\geq 2$ we clearly have $e_i^{(l)}v_j^*=0$ and 
$e_i^{(l)}\kappa(v_j^*)=0$. 
The argument for $f_i$ works similarly.  
\end{proof}
Lemma~\ref{lem:monomorphism} enables us to consider the mixed tensor
space $V^{\otimes r}\otimes {V^*}^{\otimes s}$
as a ${\mathbf U}'$-submodule $T^{r,s}$ of
$V^{\otimes r+(n-1)s}$ via an embedding which we will also denote by
$\kappa$.
Thus $\mathfrak{B}_{r,s}^n(q)$ acts on $T^{r,s}$. 

If we restrict the action of
an element of ${\mathbf U}'$ on $V^{\otimes r+(n-1)s}$
or equivalently of the 
$q$-Schur algebra $S_q(n,r+(n-1)s)$
to $T^{r,s}$, then we get an element of 
$\mathrm{End}_R(T^{r,s})$. Since the actions of ${\mathbf U}'$ and
$\mathfrak{B}_{r,s}^n(q)$ commute, this is also an element of 
$\mathrm{End}_{\mathfrak{B}_{r,s}^n(q)}(T^{r,s})$. 
Let $S_q(n;r,s):=\mathrm{End}_{\mathfrak{B}_{r,s}^n(q)}(V^{\otimes r}\otimes
{V^*}^{\otimes s})$, thus we have an algebra homomorphism
$\pi:S_q(n,r+(n-1)s)\to S_q(n;r,s)$ by restriction of the action to
$T^{r,s}\cong V^{\otimes r}\otimes
{V^*}^{\otimes s}$. Our aim is to show that $\pi$ is
surjective, for then each element of
$\mathrm{End}_{\mathfrak{B}_{r,s}^n(q)}(V^{\otimes r}\otimes 
{V^*}^{\otimes s})$ is given by the action of an element of ${\mathbf U}'$. 

\begin{lem}\label{lem:construction}
  Let $M$ be a free $R$-module with basis
  $\mathcal{B}=\{b_1,\ldots,b_l\}$ and let
  $U$ be a submodule of $M$ given by a set of linear equations on the
  coefficients with respect to the basis
  $\mathcal{B}$, i.~e.~there are elements $a_{ij}\in R$ such that 
  $U=\{\sum c_ib_i\in M:\sum_j a_{ij}c_j=0 \text{ for all }i\}$.
  Let $\{b_1^*,\ldots,b_l^*\}$ be the basis of $M^*=\mathrm{Hom}_R(M,R)$ dual
  to $\mathcal{B}$ and let
  $X$ be the submodule generated by all $\sum_j a_{ij}b_j^*$. Then
  $U\cong \left(M^*/X\right)^*$.
\end{lem}
\begin{proof}
    $\left(M^*/X\right)^*$ is isomorphic to the submodule of $M^{**}$
    given by linear forms on $M^*$ that vanish on $X$. Via the natural
    isomorphism $M^{**}\cong M$, this is isomorphic to the set of
    elements of $M$ that are annihilated by $X$. An element $m=\sum_k
    c_kb_k$ is annihilated by $X$ if and only if $0=\sum_{j,k}
    a_{ij}b_j^*(c_kb_k)= \sum_k a_{ik}c_k$ for all $i$ 
    and this is true if and
    only if $m\in U$.
\end{proof}

Note that an element $\tilde\varphi\in (M^*/X)^*$ corresponds to the
element $\varphi=\sum_i \tilde\varphi(b_i^*+X)b_i$ of $U$.  In our
case $S_q(n,m)$ and $S_q(n;r,s)$ are $R$-submodules of $R$-free
algebras, namely $\mathrm{End}_R(V^{\otimes m})$ and
$\mathrm{End}_R(V^{\otimes r}\otimes {V^*}^{\otimes s})$ resp., given
by a set of linear equations, which we will determine more precisely
in Sections \ref{sec:Aqnm} and \ref{sec:Aqnrs}.

\begin{defn}
  Let $M=\mathrm{End}_R(V^{\otimes m})$ and $U=S_q(n,m)$. Then $U$ is
  defined as the algebra  of endomorphisms commuting with a
  certain set of endomorphisms and thus is given by a system of linear
  equations on the coefficients. 
  Let $A_q(n,m)=M^*/X$ as in Lemma~\ref{lem:construction}.
  Similarly let $A_q(n;r,s) =M^*/X$ with
  $M=\mathrm{End}_R(V^{\otimes r}\otimes {V^*}^{\otimes s})$ and
  $U=S_q(n;r,s)$. 
\end{defn}
By Lemma~\ref{lem:construction} we have 
$A_q(n,m)^*= S_q(n,m)$ and $A_q(n;r,s)^*=S_q(n;r,s)$.
We will proceed as follows: We will take $m = r+(n-1)s$ and define an
$R$-homomorphism $\iota:A_q(n;r,s)\to A_q(n,r+(n-1)s)$ such that
$\iota^*=\pi: S_q(n,r+(n-1)s)\to S_q(n;r,s)$.  Then we will define an
$R$-homomorphism $\phi:A_q(n,r+(n-1)s)\to A_q(n;r,s)$ such that
$\phi\circ\iota=\mathrm{id}_{A_q(n;r,s)}$ by giving suitable bases for
$A_q(n,r+(n-1)s)$ and $A_q(n;r,s)$. Dualizing this equation, we get
$\pi\circ\phi^*=\iota^*\circ\phi^*=\mathrm{id}_{S_q(n;r,s)}$, and this
shows that $\pi$ is surjective.  Actually $A_q(n,r+(n-1)s)$ and
$A_q(n;r,s)$ are coalgebras and $\iota$ is a morphism of coalgebras,
but we do not need this for our results.

\section{$A_q(n,m)$} \label{sec:Aqnm}

The description of $A_q(n,m)$ is well known, see
e.~g.~\cite{dipperdonkin}.
Let $A_q(n)$ be the free $R$-algebra on
generators $x_{ij}$ ($1\leq i,j\leq n$) subject to the relations
\begin{eqnarray*}
  x_{ik}x_{jk}&=&qx_{jk} x_{ik} \quad\text{ if } i<j\\
  x_{ki}x_{kj}&=&qx_{kj} x_{ki} \quad\text{ if } i<j\\
  x_{ij}x_{kl}&=&x_{kl}x_{ij} \quad\text{ if } i<k \text{ and } j>l\\
  x_{ij}x_{kl}&=&x_{kl}x_{ij}+(q-q^{-1})x_{il}x_{kj}\quad\text{ if } i<k 
   \text{ and }j<l. 
\end{eqnarray*}
Note that these relations define the commutative algebra in $n^2$
commuting indeterminates $x_{ij}$ in case $q=1$.
The free algebra on the generators $x_{ij}$ is obviously graded (with
all  generators in degree 1), and since the relations are
homogeneous, this induces a grading on $A_q(n)$. Then
\begin{lem}[\cite{dipperdonkin}]
  $A_q(n,m)$ is the $R$-submodule of $A_q(n)$ of elements of
  homogeneous degree $m$. 
\end{lem}
\begin{proof}
  Since our relations of the Hecke algebra differ from those in 
  \cite{dipperdonkin} ($(T_i-q)(T_i+1)=0$ is replaced by
  $(T_i+q)(T_i-q^{-1})=0$), and thus $A_q(n,m)$ differs as well, we
  include a proof here. 

  Suppose  $\varphi$ is an endomorphism of $V^{\otimes m}$ 
  commuting with the action of
  a generator $S_i$. For convenience, we assume that $m=2$ and
  $S=S_1$.
  $\varphi$
  can be written as a linear combination of the basis elements
  $E_{(ij),(kl)}$ mapping $v_k\otimes v_l$ to $v_i\otimes v_j$, and
  all other basis elements to $0$. For the coefficient of
  $E_{(ij),(kl)}$, we write $c_{ik}c_{jl}$, so that
  $\varphi=\sum_{i,j,k,l}c_{ik}c_{jl}E_{(ij),(kl)}$. On the one hand
  we have
  \begin{align*}
    &S(\varphi(v_k\otimes v_l))=
    S\left(\sum_{i,j}c_{ik}c_{jl}v_i\otimes v_j\right)\\
    &=\sum_{i<j}c_{ik}c_{jl}v_j\otimes v_i
    +q^{-1}\sum_i c_{ik}c_{il}v_i\otimes v_i\\
    &+\sum_{i>j}c_{ik}c_{jl}(v_j\otimes v_i+(q^{-1}-q)v_i\otimes v_j)\\
    &=\sum_{i\neq j}c_{ik}c_{jl}v_j\otimes v_i
    +q^{-1}\sum_i c_{ik}c_{il}v_i\otimes v_i
    +(q^{-1}-q)\sum_{i<j}c_{jk}c_{il}v_j\otimes v_i\\
  \end{align*}

  Now, suppose that $k>l$. Then 
 \begin{align*}
    \varphi(S(v_k\otimes v_l))&=
    \varphi(v_l\otimes v_k+(q^{-1}-q)v_k\otimes v_l)\\
    &= \sum_{i,j}\left(
      c_{jl}c_{ik}+(q^{-1}-q) c_{jk}c_{il}\right)v_j\otimes v_i
  \end{align*}
  Similar formulas hold for $k=l$ and $k<l$.  
  Comparing coefficients leads to the relations given above. 
\end{proof}
$A_q(n,m)$ has a basis consisting of monomials, but it will
turn out to be more convenient for our purposes to work with a
basis of standard bideterminants (see \cite{huangzhang}). Note that
the supersymmetric quantum letterplace 
algebra in \cite{huangzhang} for
$L^-=P^-=\{1,\ldots,n\}, L^+=P^+=\emptyset$ is isomorphic 
to $A_{q^{-1}}(n)\cong A_q(n)^{\mathrm{opp}}$, and we will adjust the
results  to our situation. 

A \emph{partition} $\lambda$ of $m$ is a sequence
$\lambda=(\lambda_1,\lambda_2,\ldots, \lambda_k)$ of nonnegative
integers such that $\lambda_1\geq \lambda_2\geq \ldots \geq
\lambda_k$ and  $\sum_{i=1}^k\lambda_i=m$. Denote the set of
partitions of $m$ by $\Lambda^+(m)$. 
The \emph{Young diagram} $[\lambda]$ 
of a partition $\lambda$ is $\{(i,j)\in
\mathbb{N}\times \mathbb{N}:1\leq i\leq k,1\leq j\leq
\lambda_i\}$. It can be represented by an
array of boxes, $\lambda_1$ boxes in the first row, $\lambda_2$ boxes
in the second row, etc.

A \emph{$\lambda$-tableau} $\mathfrak{t}$ 
is a map  $f:[\lambda]\to \{1,\ldots,n\}$. A
tableau can be represented by writing  the entry $f(i,j)$ into the
$(i,j)$-th box. A tableau $\mathfrak{t}$ is called \emph{standard},
if the entries
in each row are strictly increasing from left to right,
and the entries in each column are nondecreasing downward. In the
literature, this property is also called semi-standard, and the role
of rows and
columns may be interchanged. Note that if $\mathfrak{t}$ is a standard
$\lambda$-tableau, then $\lambda_1\leq n$. 
 A pair 
$[\mathfrak{t},\mathfrak{t}']$ of $\lambda$-tableaux is
called a \emph{bitableau}. It is standard 
if both $\mathfrak{t}$ and $\mathfrak{t}'$ are
standard $\lambda$-tableaux.

Note that the next definition differs from
the definition in \cite{huangzhang} by a sign. 
 \begin{defn}
   Let  $i_1,\ldots ,i_k,j_1,\ldots,j_k$ be  elements of $\{1,\ldots,n\}$,
   For $i_1<i_2<\ldots<i_k$ 
   let the \emph{right quantum minor} be defined by
   \[
   (i_1i_2\ldots i_k|j_1j_2\ldots j_k)_r:=
   \sum_{w\in\mathfrak{S}_k}(-q)^{l(w)}
   x_{i_{w1}j_{1}}x_{i_{w2}j_{2}}\ldots x_{i_{wk}j_{k}}.
   \]
   For arbitrary  $i_1,\ldots ,i_k$, the right 
   quantum minor is then defined by the rule 
   \[ (i_1\ldots i_li_{l+1}\ldots i_k|j_1j_2\ldots j_k)_r:=
   -q^{-1}(i_1\ldots i_{l-1}i_{l+1}i_li_{l+2}
   \ldots i_k|j_1j_2\ldots j_k)_r\]
   for $i_l>i_{l+1}$. 
   Similarly, let the \emph{left quantum minor} be defined by 
   \begin{eqnarray*}
     (i_1\ldots i_k|j_1\ldots j_k)_l&:=& 
     \sum_{w\in\mathfrak{S}_k}(-q)^{l(w)}
     x_{i_1,j_{w1}}x_{i_2j_{w2}}\ldots x_{i_kj_{wk}}
     \text{ if }j_1<\ldots<j_k,\\
     (i_1\ldots i_k|j_1\ldots j_k)_l&:=&-q^{-1} 
     (i_1\ldots i_k|j_1\ldots j_{l+1}j_l\ldots j_k)_l
     \text{ if }j_l>j_{l+1}.
   \end{eqnarray*}
   Finally let  the 
   \emph{quantum determinant} be defined by
   \[{\det}_q:=(12\ldots n|12\ldots n)_r
   =(12\ldots n|12\ldots n)_l.\]
   If $[\mathfrak{t},\mathfrak{t}']$ is a bitableau,
   and $\mathfrak{t}_1,\mathfrak{t}_2,\ldots,\mathfrak{t}_k$  
   resp.~ $\mathfrak{t}_1',\mathfrak{t}_2',\ldots,\mathfrak{t}_k'$ are
   the rows of $\mathfrak{t}$ resp.~$\mathfrak{t}'$, then let
   \[
   (\mathfrak{t}|\mathfrak{t}'):=
   (\mathfrak{t}_k|\mathfrak{t}_k')_r\ldots 
   (\mathfrak{t}_2|\mathfrak{t}_2')_r
   (\mathfrak{t}_1|\mathfrak{t}_1')_r.
   \]
   $(\mathfrak{t}|\mathfrak{t}')$ is called a \emph{bideterminant}.
\end{defn}

\begin{rem}\label{rem:properties}
  We note the following properties of  quantum minors:
\begin{enumerate}
\item 
\begin{eqnarray*}
  (i_1\ldots i_k|j_1\ldots j_k)_r&=&
  -q (i_1\ldots i_k|j_1\ldots j_{l+1}j_l\ldots j_k)_r
  \text{ for }j_l>j_{l+1}\\ 
  (i_1\ldots i_k|j_1\ldots j_k)_l&=&
  -q(i_1\ldots i_{l+1}i_l\ldots i_k|j_1\ldots j_k)_l
  \text{ for }i_l>i_{l+1}.
\end{eqnarray*}

\item If
  $i_1<i_2<\ldots<i_k$ and $j_1<j_2<\ldots<j_k$, then right and left
  quantum minors coincide, and we simply write $(i_1\ldots
  i_k|j_1\ldots j_k)$. This notation thus indicates
  that the sequences of numbers
  are increasing. In general, right and left quantum minors differ by
  a power of $-q$.
\item If two $i_l$'s or $j_l$'s coincide, then
  the quantum minors vanish.
\item \label{item:quantum_determinant_center}
  The quantum determinant $\det_q$ is an
  element of the center of $A_q(n)$. 
\end{enumerate}
\end{rem}

\begin{defn}\label{defn:content}
  Let the \emph{content} of a monomial $x_{i_1j_1}\ldots x_{i_{m}j_{m}}$
  be defined as the tuple $(\alpha,\beta)=
  ((\alpha_1,\ldots,\alpha_n),(\beta_1,\ldots,\beta_n))$
  where $\alpha_i$
  is the number of indices $i_t$ such that $i_t=i$, and   $\beta_j$
  is the number of indices $j_t$ such that $j_t=j$. Note that
  $\sum\alpha_i=\sum\beta_j=m$ for each monomial of homogeneous
  degree $m$. For such a tuple $(\alpha,\beta)$,
  let $P(\alpha,\beta)$ be the subspace of $A_q(n,m)$ generated by the
  monomials of content $(\alpha,\beta)$. 
  Furthermore, let the \emph{content} of a bitableau
  $[\mathfrak{t},\mathfrak{t}']$
  be defined similarly as the
  tuple $(\alpha,\beta)$, such that $\alpha_i$ is the number of
  entries in $\mathfrak{t}$ equal to $i$ and  $\beta_j$ is the number of
  entries in $\mathfrak{t}'$ equal to $j$.
\end{defn}

\begin{thm}[\cite{huangzhang}]\label{thm:standardbasis}
  The bideterminants $(\mathfrak{t}|\mathfrak{t}')$ of the 
  standard $\lambda$-tableaux with  $\lambda$
  a partition of $m$ form a basis of $A_q(n,m)$, such that the
  bideterminants of standard $\lambda$-tableaux of content
  $(\alpha,\beta)$ form a basis of $P(\alpha,\beta)$.
\end{thm}
The proof  in \cite{huangzhang} works over a field,
but the arguments
are valid if the field is replaced by a commutative  ring with $1$.
The reversed order of the minors is due to the opposite algebra. Note
that for $i_1<i_2<\ldots<i_k$ and $j_1<j_2<\ldots<j_k$ we have 
\[ q^{\frac{k(k-1)}{2}}
(i_1i_2\ldots i_k|j_1j_2\ldots j_k)_r=
\sum_{w\in\mathfrak{S}_k}(-q)^{-l(w)}
   x_{i_{wk}j_{1}}x_{i_{w(k-1)}j_{2}}\ldots x_{i_{w1}j_{k}},
\]
which is a quantum minor of $A_{q^{-1}}(n)^{\mathrm{opp}}$. 

\begin{lem}[Laplace's expansion \cite{huangzhang}]

  \begin{enumerate}
  \item  For $j_1<j_2<\ldots <j_l<j_{l+1}<\ldots <j_k$ we have
    \begin{align*}
      &(i_1i_2\ldots i_k|j_1j_2\ldots j_k)_l\\
      &=\sum_w (-q)^{l(w)}
      (i_1\ldots i_l|j_{w1}\ldots j_{wl})_l
      (i_{l+1}\ldots i_k|j_{w(l+1)}\ldots j_{wk})_l
    \end{align*} 
    where the summation is over all $w\in
    \mathfrak{S}_k $, such that $w1<w2<\ldots
    <wl$ and $w(l+1)<w(l+2)<\ldots <wk$.
\item
  For 
  $i_1<i_2<\ldots  <i_k$ we have
  \begin{align*}
    &(i_1i_2\ldots i_k|j_1j_2\ldots j_k)_r\\
    &=\sum_w (-q)^{l(w)}
    (i_{w1}\ldots i_{wl}|j_{1}\ldots j_{l})_r
    (i_{w(l+1)}\ldots i_{wk}|j_{l+1}\ldots j_{k})_r
  \end{align*} 
  the summation again over all $w\in
  \mathfrak{S}_k $, such that $w1<w2<\ldots
  <wl$ and $w(l+1)<w(l+2)<\ldots <wk$.
\end{enumerate}
\end{lem}

\section{$A_q(n;r,s)$} \label{sec:Aqnrs}
A basis of $\mathrm{End}_R(V^{\otimes r}\otimes {V^*}^{\otimes s})$ is
given by matrix units $E_{\mathbf i|\mathbf j\;\mathbf k|\mathbf l}$
such that $E_{\mathbf i|\mathbf j\;\mathbf k|\mathbf l}v_{\mathbf
  s|\mathbf t} = \delta_{\mathbf k|\mathbf l,\mathbf s|\mathbf
  t}v_{\mathbf i|\mathbf j}$.  Suppose $\varphi=\sum\limits_{\mathbf
  i,\mathbf j, \mathbf k, \mathbf l} c_{\mathbf i|\mathbf j\;\mathbf
  k|\mathbf l} E_{\mathbf i|\mathbf j\;\mathbf k|\mathbf l} \in
\mathrm{End}_R(V^{\otimes r}\otimes {V^*}^{\otimes s})$ commutes with
the action of $\mathfrak{B}_{r,s}^n(q)$ or equivalently with a set of
generators of $\mathfrak{B}_{r,s}^n(q)$.  Since coefficient spaces are
multiplicative, we can write
\[
c_{i_1k_1}c_{i_2k_2}\ldots c_{i_rk_r}
c_{j_1l_1}^*c_{j_2l_2}^*\ldots c_{j_sl_s}^*
\]
for the coefficient $c_{\mathbf i|\mathbf j\;\mathbf k|\mathbf l}$.
It is easy to see from the description of $A_q(n,m)$ that $\varphi$
commutes with the generators without non-propagating edges if and only
if the $c_{ij}$ satisfy the relations of $A_q(n)$ and the $c_{ij}^*$
satisfy the relations of $A_{q^{-1}}(n)\cong A_q(n)^{\mathrm{opp}}$.

Now suppose that $\varphi$ in addition commutes with the action of the
generator 
\[e=\raisebox{-.1cm}{\epsfbox{si.20}}.\] 
We assume that $r=s=1$ (the general case being similar) and 
$\varphi=\sum\limits_{i,j,k,l=1}^n
c_{ik}c_{jl}^*E_{i|j\;k|l}$. 
 Let $v=v_i\otimes v_j^*$ be a basis element of
 $V\otimes V^*$.  We have (the indices in the sums always run 
 from $1$ to $n$)
\begin{eqnarray*}
  \varphi(v)e&=& \sum_{s, t} c_{si}c_{tj}^*(v_s\otimes v_t^*)e
  =  \sum_{s,k}q^{2s-n-1} c_{si}c_{sj}^*(v_k\otimes v_k^*)\\
   \varphi(ve)&=& \delta_{i j}q^{2i-n-1}\sum_{k}\varphi(v_k\otimes v_k^*)
  =\delta_{i j}q^{2i-n-1}\sum_{k,s,t}
  c_{sk}c_{tk}^*v_s\otimes v_t^*
\end{eqnarray*}

Comparing coefficients, we get the following conditions: 
\begin{gather*}
  \sum_{k=1}^n c_{ik}c_{jk}^*=0 \text{ for }i\neq j\\
  \sum_{k=1}^n q^{2k}c_{ki}c_{kj}^*=0 \text{ for }i\neq j\\
  \sum_{k=1}^n q^{2k-2i}c_{ki}c_{ki}^*=\sum_{k=1}^n c_{jk}c_{jk}^*.
\end{gather*}
This, combined with Lemma~\ref{lem:construction} shows that

\begin{lem}
  \[  A_q(n;r,s)\cong (F(n,r)\otimes_R F_*(n,s))/Y\]
  where $F(n,r)$ resp.~$F_*(n,s)$  is the $R$-submodule
   of the free algebra on generators $x_{ij}$ resp.~$x_{ij}^*$
   generated by monomials of degree
   $r$ resp.~$s$ and
   $Y$ is the $R$-submodule of 
   $F(n,r)\otimes_R F_*(n,s)$ generated by elements of the form
   $h_1h_2h_3$ where $h_2$ is one of the elements
  \begin{gather} 
    x_{ik}x_{jk}-qx_{jk} x_{ik}\text{ for }i<j
    \label{equ:A(nrs):1}\\
    x_{ki}x_{kj}-qx_{kj} x_{ki}\text{ for }i<j
    \label{equ:A(nrs):2}\\
    x_{ij}x_{kl}-x_{kl}x_{ij}\text{ for }i<k,j>l
    \label{equ:A(nrs):3}\\
    x_{ij}x_{kl}-x_{kl}x_{ij}-(q-q^{-1})x_{il}x_{kj}\text{ for
    }i<k,j<l
    \label{equ:A(nrs):4}\\
    x_{ik}^*x_{jk}^*-q^{-1}x_{jk}^* x_{ik}^*\text{ for }i<j
    \label{equ:A(nrs):5}\\
    x_{ki}^*x_{kj}^*-q^{-1}x_{kj}^* x_{ki}^*\text{ for }i<j
    \label{equ:A(nrs):6}\\
    x_{ij}^*x_{kl}^*-x_{kl}^*x_{ij}^*\text{ for }i<k,j>l
    \label{equ:A(nrs):7}\\
    x_{ij}^*x_{kl}^*-x_{kl}^*x_{ij}^*
    +(q-q^{-1})x_{il}^*x_{kj}^*\text{ for }i<k,j<l
    \label{equ:A(nrs):8}\\
    \sum_{k=1}^n x_{ik}x_{jk}^* \text{ for }i\neq j
    \label{equ:A(nrs):9}\\
    \sum_{k=1}^n q^{2k}x_{ki}x_{kj}^* \text{ for }i\neq j
    \label{equ:A(nrs):10}\\
    \sum_{k=1}^n q^{2k-2i}x_{ki}x_{ki}^*-\sum_{k=1}^n x_{jk}x_{jk}^* 
    \label{equ:A(nrs):11}
  \end{gather}
  and $h_1,h_3$ are monomials of appropriate degree.
\end{lem}

\begin{rem}\label{rem:automorphism}
  The map given by
  $x_{ik}\mapsto q^{2k-2i}x_{ki}$ and $x_{ik}^*\mapsto x_{ki}^*$
  induces an $R$-linear automorphism of $A_q(n;r,s)$.
\end{rem}

Bideterminants can also be formed using the variables $x_{ij}^*$. In
this case let
\[
(\mathfrak{t}|\mathfrak{t}')^*:=(\mathfrak{t}_1|\mathfrak{t}_1')_r^*
(\mathfrak{t}_2|\mathfrak{t}_2')_r^*\ldots
(\mathfrak{t}_k|\mathfrak{t}_k')_r^*\]
where the quantum minors
$(i_1\ldots i_k|j_1\ldots j_k)_{r/l}^*$ are defined as above with $q$
replaced by $q^{-1}$.

\section{The map $\iota:A_q(n;r,s)\to A_q(n,r+(n-1)s)$}\label{sec:iota}
For any $1 \le i,j \le n$ let $\iota(x_{ij})=x_{ij}$ and 
\[\iota(x_{ij}^*)=
(-q)^{j-i}(12\ldots\hat{i}\ldots n|
12\ldots\hat{j}\ldots n)
\in A_q(n,n-1),\]
then there is a unique $R$-linear map
\[
\iota:F(n,r)\otimes_R F_*(n,s)\to
A_q(n,r+(n-1)s)
\] 
such that $\iota(x_{i_1j_1}\ldots x_{i_rj_r}
x_{k_1l_1}^*\ldots x_{k_sl_s}^* )=\iota(x_{i_1j_1})\ldots \iota(x_{i_rj_r})
\iota(x_{k_1l_1}^*)\ldots \iota(x_{k_sl_s}^*)$. 
\begin{lem}\label{lem:factor}
  $Y$ is contained in the kernel of $\iota$ and thus $\iota$ 
  induces an $R$-linear map 
  \[A_q(n;r,s)\to A_q(n,r+(n-1)s)\]
which we will then also denote by $\iota$. 
\end{lem}

\begin{proof}
  We have to show that the generators of $Y$ lie in the kernel of
  $\iota$.  Generators of $Y$ involving the elements
  \eqref{equ:A(nrs):1} up to \eqref{equ:A(nrs):4} are obviously in the
  kernel of $\iota$.  {\cite[Theorem~7.3]{goodearl}} shows that
  generators involving elements \eqref{equ:A(nrs):5} up to
  \eqref{equ:A(nrs):8} are also in the kernel.  Laplace's Expansion
  shows that
  \begin{eqnarray*}
    \iota\left(\sum_{k=1}^n x_{ik} x_{jk}^*\right)&=&
    \sum_{k=1}^n(-q)^{(k-1)-(j-1)}x_{ik}\cdot
    (1\ldots\hat{j}\ldots n|1\ldots\hat{k}\ldots n)_l\\
    &=&(-q)^{1-j}  (i1\ldots\hat{j}\ldots n|1\ldots n)_l
    = \delta_{i,j}\cdot{\det}_q \text{ and }\\ 
    \iota\left(\sum_{k=1}^n q^{2k-2i}x_{ki} x_{kj}^*\right)&=&
    q^{-2i+j+1}\sum_{k=1}^n(-q)^{k-1}x_{ki}\cdot
    (1\ldots\hat{k}\ldots n|1\ldots\hat{j}\ldots n)_r\\
    &=&(-q)^{j-2i+1}  (1\ldots n|i1\ldots\hat{j}\ldots n)_r
    = \delta_{i,j}\cdot{\det}_q,
\end{eqnarray*}
thus the generators involving the elements \eqref{equ:A(nrs):9} 
  up to \eqref{equ:A(nrs):11} are  in the kernel of $\iota$. 
\end{proof}

Now, we have maps 
\[
\iota^*: A_q(n,r+(n-1)s)^*\to A_q(n;r,s)^*\mbox{ and }
\pi:S_q(n,r+(n-1)s)\to S_q(n;r,s).\] By definition
$A_q(n,r+(n-1)s)^*\cong S_q(n,r+(n-1)s)$ and $A_q(n;r,s)^*\cong
S_q(n;r,s)$. Under these identifications we have

\begin{lem}
  $\iota^*=\pi.$
\end{lem}

\begin{proof}
  We will write 
  \begin{eqnarray*}
    x_{i_1\ldots i_l\;j_1\ldots j_l}&=&x_{i_1,j_1}\ldots x_{i_l,j_l} 
    \text{ and }\\
    x_{i_l\ldots i_1|l_1\ldots l_m\;j_l\ldots j_1|k_1\ldots
      k_m}&=&x_{i_l,j_l}\ldots x_{i_1,j_1} x_{l_1,k_1}^*\ldots
    x_{l_m,k_m}^*.
  \end{eqnarray*}
  Suppose that $\tilde\varphi\in A_q(n,r+(n-1)s)^*$. The
  corresponding element of $S_q(n,r+(n-1)s)$ is $\varphi=\sum\limits_{\mathbf
    i,\mathbf j\in \mathrm{I}(n,r+(n-1)s)} \tilde\varphi(x_{\mathbf
    i\mathbf j})E_{\mathbf i\mathbf j}$. Since
  $\iota^*(\tilde\varphi)=\tilde\varphi\circ\iota$, we have 
  \[\iota^*(\varphi)=
  \sum_{\mathbf i,\mathbf j,\mathbf k,\mathbf l}
  \tilde \varphi\circ\iota(x_{\mathbf i|\mathbf j\,\mathbf k|\mathbf l})
  E_{\mathbf i|\mathbf j\,\mathbf k|\mathbf l}\] 
  In other words: The coefficient of  
  $E_{\mathbf i|\mathbf j\,\mathbf k|\mathbf l} $ in
  $\iota^*(\varphi)$ can be computed by substituting each $x_{\mathbf
    s\mathbf t}$ in
  $\iota(x_{\mathbf i|\mathbf j\,\mathbf k|\mathbf l})$ by $\tilde\varphi(x_{\mathbf
    s\mathbf t})$.
  On the other hand, to compute the coefficient of 
  $E_{\mathbf i|\mathbf j\,\mathbf k|\mathbf l} $ in
  $\pi(\varphi)$, one has to
  consider the action of $\varphi$ on
  a basis element $v=\kappa(v_{\mathbf k|\mathbf l})$ of $T^{r,s}$. 
  For a multi index $\mathbf l\in\mathrm{I}(n,s)$ let $\mathbf
  l^*\in\mathrm{I}(n,(n-1)s)$ be defined by 
  \[\mathbf  l^*=(1\ldots\hat{l_1}\ldots n1\ldots\hat{l_2}\ldots
  n\ldots 1\ldots\hat{l_s}\ldots n ).\] 
  Then \[v=\kappa(v_{\mathbf k|\mathbf l})
  =(-q)^{l_1+l_2+\ldots +l_s}\sum_{w\in\mathfrak{S}_{n-1}^{\times
      s}}(-q)^{l(w)}
  v_{\mathbf k}\otimes(v_{\mathbf l^*}
  T_w)\]
  and thus we have
  \begin{eqnarray*}
    \varphi(v)&=&(-q)^{\sum l_k}\sum_{\mathbf s,\mathbf t,w}
    (-q)^{l(w)}\tilde\varphi(x_{\mathbf s\mathbf t})
    E_{\mathbf s\mathbf t}\left(v_{\mathbf k}\otimes (v_{\mathbf l^*}T_w)\right)\\
    &=&\sum_{\mathbf s,w}
    (-q)^{l(w)+\sum l_k}\tilde\varphi(x_{\mathbf s\;\mathbf k \mathbf l^*\!\!.w})
    v_{\mathbf s}.
  \end{eqnarray*}

  Since $\varphi$ leaves $T^{r,s}$ invariant, $\varphi(v)$ is a
  linear combination of the basis elements $\kappa(v_{\mathbf
    i|\mathbf j})$ of $T^{r,s}$. Distinct 
  $\kappa(v_{\mathbf i|\mathbf j})$ involve distinct basis vectors of
  $V^{\otimes r+(n-1)s}$. Thus if $\varphi(v)=
  \sum_{\mathbf i|\mathbf j} 
  \lambda_{\mathbf i|\mathbf j}\kappa(v_{\mathbf i|\mathbf j})=
  \sum_{\mathbf i|\mathbf j,w} 
  \lambda_{\mathbf i|\mathbf j}(-q)^{l(w)+j_1+\ldots+j_s}
  v_{\mathbf i\mathbf j^*.w}$  then
  $(-q)^{\sum j_k}\lambda_{\mathbf i|\mathbf j}$
  is equal to the coefficient of
  $v_{\mathbf i \mathbf j^*}$
  when $\varphi(v)$ is written as a linear
  combination of basis vectors of $V^{\otimes r+(n-1)s}$.
  The coefficient of 
  $v_{\mathbf i \mathbf j^*}$ in $\varphi(v)$ is, by the formula above, 
  \[(-q)^{\sum l_k}
  \sum_w(-q)^{l(w)}\tilde
  \varphi(x_{\mathbf i \mathbf j^*\;\mathbf k \mathbf l^*\!\!.w}).\]
  Thus 
  \[
  \lambda_{\mathbf i|\mathbf j}=
  (-q)^{\sum l_k-j_k}
  \sum_w(-q)^{l(w)}\tilde
  \varphi(x_{\mathbf i \mathbf j^*\;\mathbf k \mathbf l^*\!\!.w}) 
  =\tilde\varphi\circ\iota 
  (x_{\mathbf i|\mathbf j\;\mathbf k|\mathbf l}).
  \]
  But
  $\lambda_{\mathbf i|\mathbf j}$ is also the coefficient of
  $E_{\mathbf i|\mathbf j\;\mathbf k|\mathbf l }$ in $\pi(\varphi)$
  which shows the result.
\end{proof}

\begin{thm}[Jacobi's Ratio Theorem]\label{thm:jacobi}
  Suppose $n\geq l\geq 0$, 
  and $i_1<i_2<\ldots <i_l$ and $j_1<j_2<\ldots
  <j_l$. 
  Let $i_1'<i_2'<\ldots <i_{n-l}'$ and $j_1'<j_2'<\ldots <j_{n-l}'$
  be the unique numbers such that $\{1,\ldots,n\}=
  \{i_1,\ldots,i_l,i_1',\ldots,i_{n-l}'\}=
  \{j_1,\ldots,j_l,j_1',\ldots,j_{n-l}'\}$. Then 
  \[
  \iota\left((i_1\ldots i_l|j_1\ldots j_l)^*\right)
  =(-q)^{\sum_{t=1}^l (j_t-i_t)}\mathrm{det}_q^{l-1}
  (i_1' \ldots i_{n-l}'|j_1'\ldots j_{n-l}').
  \]
\end{thm}
\begin{proof}
  We argue by induction on $l$. 
  Note that for $l=0$, $\mathrm{det}_q^{l-1}=\mathrm{det}_q^{-1}$ is not
  an element of $A_q(n)$. However  $(i_1' \ldots i_{n-l}'|j_1'\ldots
  j_{n-l}')$ turns out to be $\det_q$, thus the right hand side of the
  formula is $\mathrm{det}_q^{-1}\det_q=1=\iota(1)$. In this sense the
  formula is  valid for $l=0$.   
  
  For $l=1$ the theorem is true by the definition of
  $\iota(x_{ij}^*)$. Now assume
  the theorem is true for $l-1$. Apply Laplace's expansion and use
  induction to get 
  \begin{align*}
    \iota&\left((i_1\ldots i_l|j_1\ldots j_l)^*\right)
    =\iota\left(\sum_{k=1}^l(-q)^{-(k-1)}x_{i_kj_1}^* 
      (i_1\ldots\hat{i_k}\ldots i_l|j_2\ldots\ldots j_l)^*
    \right)\\
    &= \sum_{k=1}^l(-q)^{1-k}(-q)^{j_1-i_k}
    (1\ldots\hat{i_k}\ldots n|1\ldots\hat{j_1}\ldots n)\cdot
    (-q)^{\sum\limits_{t\neq1}j_t-\sum\limits_{t\neq k}i_t} 
    \mathrm{det}_q^{l-2}\\
    &\cdot(1\ldots\hat{i_1}\ldots\hat{i_2}\ldots\ldots\hat{i_{k-1}}
    \ldots\hat{i_{k+1}}\ldots\ldots\hat{i_l}\ldots n|
    1\ldots\hat{j_2}\ldots\hat{j_3}\ldots\ldots\hat{j_l}\ldots n)
  \end{align*}
  We claim that this is equal to 
  \begin{align}
    (-q)^{\sum_{t=1}^l (j_t-i_t)}& \mathrm{det}_q^{l-2}
    \sum_w(-q)^{l(w)+1-n}
    (w1w2\ldots w(n-1)|1\ldots\hat{j_1}\ldots n)\notag\\
    & \cdot(wn\,1\ldots\hat{i_1}\ldots\ldots\hat{i_l}\ldots n|
    1\ldots\hat{j_2}\ldots\ldots\hat{j_l}\ldots n)_l\label{term:expansion}
  \end{align}
  where the summation is over all $w\in \mathfrak{S}_n$ such that
  $w1<w2<\ldots <w(n-1)$. If $wn$ is not one of the $i_k$'s, then the
  summand in (\ref{term:expansion}) vanishes, since $wn$ appears 
  twice in the row on the left
  side of the second minor. Thus the summation is over all $w$
  as above with $wn=i_k$ for some $k$. Note that $l(w)=n-i_k$ and 
  \[(i_k1\ldots\hat{i_1}\ldots\ldots\hat{i_l}\ldots n|\mathfrak{t})_l
  =(-q)^{i_k-k}(1\ldots\hat{i_1}\ldots\hat{i_{k-1}}
  \ldots\hat{i_{k+1}}\ldots\hat{i_l}\ldots n|
  \mathfrak{t}),
  \]
  the claim follows. 
  Again apply Laplace's expansion to the second minor in
  (\ref{term:expansion}) to get
  \begin{gather*}
    (wn\,1\ldots\hat{i_1}\ldots\ldots\hat{i_l}\ldots n|
    1\ldots\hat{j_2}\ldots\ldots\hat{j_l}\ldots n)_l\\
    =\sum_v(-q)^{l(v)}x_{wn\;v1} 
    (1\ldots\hat{i_1}\ldots\ldots\hat{i_l}\ldots n|
    v2v3\ldots\hat{vj_2}\ldots\ldots\hat{vj_l}\ldots vn),
  \end{gather*}
  the summation being over all $v\in
  \mathfrak{S}_{\{1,\ldots,\hat{j_2},\ldots,\hat{j_l},\ldots,n\}}$ with
  $v2<v3<\ldots <vn$. 
  After substituting this term in (\ref{term:expansion}), one can again
  apply Laplace's expansion, to get that
  (\ref{term:expansion}) is equal to 
  \begin{align}
    (-q)^{\sum(j_t-i_t)}& \mathrm{det}_q^{l-2}
    \sum_v(-q)^{l(v)+1-n} 
    (12\ldots n|1\ldots\hat{j_1}\ldots n\,v1)_r\notag\\
    & \cdot 
    (1\ldots\hat{i_1}\ldots\ldots\hat{i_l}\ldots n|
    v2v3\ldots\hat{vj_2}\ldots\ldots\hat{vj_l}\ldots vn)
    \label{term:expansion2}
  \end{align}
  The only summand in (\ref{term:expansion2}) that does not vanish, is
  the term for $v1=j_1$ with $l(v)=j_1-1$.
  Thus (\ref{term:expansion2}) is equal to 
  \begin{align*}
    (-q&)^{\sum(j_t-i_t)}  \mathrm{det}_q^{l-2}
    (-q)^{j_1-n} 
    (12\ldots n|1\ldots\hat{j_1}\ldots nj_1)_r
    \cdot (i_1'\ldots i_{n-l}'|j_1'\ldots j_{n-l}')\\
    &=(-q)^{\sum_{t=1}^l (j_t-i_t)}\mathrm{det}_q^{l-1}
    (i_1'\ldots i_{n-l}'|j_1'\ldots j_{n-l}').
  \end{align*}
\end{proof}

\section{A basis for $A_q(n;r,s)$}\label{section:basis}
Theorem~\ref{thm:jacobi} enables us to construct elements of
$A_q(n;r,s)$ that are mapped to standard bideterminants under $\iota$.
First, we will introduce the notion of rational tableaux, although we
will slightly differ from the definition of rational tableaux in
\cite{stembridge}. Recall that $\Lambda^+(k)$ is the set of partitions
of $k$. 
\begin{defn}\label{defn:rationaltableaux}
  Fix $0\leq k\leq \mathrm{min}(r,s)$. Let  
  $\rho\in \Lambda^+(r-k)$ and $\sigma\in
  \Lambda^+(s-k)$ with $\rho_1+\sigma_1\leq n$.
  A \emph{rational $(\rho,\sigma)$-tableau} is a pair
  $(\mathfrak{r},\mathfrak{s})$ with $\mathfrak{r}$
  a  $\rho$-tableau  and $\mathfrak{s}$  a
  $\sigma$-tableau.
  
  Let $\mathrm{first}_i(\mathfrak{r},\mathfrak{s})$ be the number 
  of entries of the first row of $\mathfrak{r}$ which are $\leq i$
  plus the number 
  of entries of the first row of $\mathfrak{s}$ which are $\leq i$. 
  A rational tableau is called \emph{standard} if 
  $\mathfrak{r}$ and $\mathfrak{s}$ are standard tableaux
  and the following condition  holds:
  \begin{equation}
    \mathrm{first}_i(\mathfrak{r},\mathfrak{s})\leq i 
    \text{ for all }i=1,\ldots,n
    \label{equ:condition}
  \end{equation}

  A pair $[(\mathfrak{r},\mathfrak{s}),
  (\mathfrak{r}',\mathfrak{s}')]$ of rational
  $(\rho,\sigma)$-tableaux is
  called a \emph{ra\-tio\-nal bi\-ta\-bleau},
  and it is called a
  standard rational bitableau if both 
  $(\mathfrak{r},\mathfrak{s})$ and 
  $(\mathfrak{r}',\mathfrak{s}')$ are
  standard rational tableaux. 
\end{defn}

\begin{rem}
  In \cite{stembridge}, condition~\eqref{equ:condition} is already
  part of the definition of rational tableaux.  The condition
  $\rho_1+\sigma_1\leq n$ is equivalent to
  condition~\eqref{equ:condition} for $i=n$. The reason for the
  difference will be apparent in the proof of the next lemma. 
\end{rem}

\begin{lem}\label{lem:bijection}
  There is a bijection between the set consisting of all standard
  rational $(\rho,\sigma)$-tableaux for $\rho\in \Lambda^+(r-k)$,
  $\sigma\in \Lambda^+(s-k)$, as $k$ runs from 0 to $\min(r,s)$ and
  the set of all standard $\lambda$-tableaux for
  $\lambda\in\Lambda^+(r+(n-1)s)$ satisfying $\sum_{i=1}^s\lambda_i
  \geq (n-1)s$.
\end{lem}

\begin{proof}
  Given a rational $(\rho,\sigma)$-tableau 
  $(\mathfrak{r},\mathfrak{s})$ we construct a
  $\lambda$-tableau $\mathfrak{t}$ as follows: 
  Draw a rectangular diagram with $s$ rows and $n$ columns. Rotate the
  tableau $\mathfrak{s}$ by 180 degrees and place 
  it in the bottom right corner
  of the  rectangle. Place the tableau $\mathfrak{r}$
  on the left side
  below the
  rectangle. Fill the empty boxes of the rectangle with numbers, such
  that in each row, the entries that do not appear in 
  $\mathfrak{t}$ appear in
  the empty boxes in increasing order. Let $\mathfrak{t}$ be the tableau
  consisting of the formerly empty boxes
  and the boxes of $\mathfrak{r}$.  We
  illustrate this procedure with an example. Let
  $n=5,r=4,s=5,k=1$
  and let \Yvcentermath1
  \[ 
  (\mathfrak{r},\mathfrak{s})=
  \left(\young(13,2),\young(34,35)
  \right). 
  \]
 Then \Yvcentermath1
  \[ 
  (\mathfrak{r},\mathfrak{s})
  \leadsto\young(\hfill\hfill\hfill\hfill\hfill 
  ,\hfill\hfill\hfill\hfill\hfill 
  ,\hfill\hfill\hfill\hfill\hfill
  ,\hfill\hfill\hfill 53,\hfill\hfill\hfill 43,13,2)
  \leadsto
  \young(12345,12345,12345,12453,12543,13,2)
  \leadsto \mathfrak{t}= 
  \young(12345,12345,12345,124,125,13,2)
  \]
  It is now easy to give an inverse: Just draw the rectangle into the
  tableau $\mathfrak{t}$, fill the empty boxes of the rectangle in a
  similar way as before, rotate these back to obtain
  $\mathfrak{s}$. $\mathfrak{r}$ is the part of the tableau
  $\mathfrak{t}$, that lies outside the rectangle.  We have to show,
  that these bijections provide standard tableaux of the right shape.

  Suppose $(\mathfrak{r},\mathfrak{s})$ 
  is a rational $(\rho,\sigma)$-tableau,
  then $\mathfrak{t}$ is a $\lambda$-tableau,
  with $\lambda_i=n-\sigma_{s+1-i}$
  for $i\leq s$ and $\lambda_i=\rho_{i-s}$ for $i>s$. Thus
  $\lambda_i\geq \lambda_{i+1}$ for $i<s$ is equivalent to 
  $\sigma_{s+1-i}\leq \sigma_{s-i}$, and for $i>s$ it is equivalent to
  $\rho_{i-s}\geq \rho_{i+1-s}$. Now  $\rho_1+\sigma_1=
  \lambda_{s+1}-(\lambda_{s}-n)$. This shows that $\lambda$ is a
  partition if and only if $\rho$ and $\sigma$ are partitions with
  $\rho_1+\sigma_1\leq n$.
  We still have to show that 
  $(\mathfrak{r},\mathfrak{s})$ is standard if and only
  if $\mathfrak{t}$ is standard. 

  By definition, all standard tableaux have increasing
  rows. A tableau has nondecreasing columns if and only if for all
  $i=1,\ldots, n$ and all rows (except for the last row)
  the number of  entries $\leq i$ in this row is greater or equal than 
  the number of  entries $\leq i$ in the next row. Now, it follows
  from the construction that $\mathfrak{t}$ 
  has nondecreasing columns inside the
  rectangle if and only if $\mathfrak{s}$ has  nondecreasing columns,
  $\mathfrak{t}$ has nondecreasing columns outside the
  rectangle if and only if $\mathfrak{r}$
  has  nondecreasing columns, and the
  columns in $\mathfrak{t}$ do not decrease
  from row $s$ to row $s+1$ if and only
  if condition~\eqref{equ:condition} holds. 
\end{proof}

\begin{defn}
  For $ k\geq 1$ let $\mathfrak{det}_q^{(k)}
  \in A_q(n;k,k)$ be recursively defined by
  $\mathfrak{det}_q^{(1)}:=\sum_{l=1}^n x_{1l}x_{1l}^*$ and
  $\mathfrak{det}_q^{(k)}:=
  \sum_{l=1}^n x_{1l}\mathfrak{det}_q^{(k-1)}x_{1l}^* $ for $k>1$.

  Let a (rational) bideterminant 
  $\left((\mathfrak{r},\mathfrak{s})|
    (\mathfrak{r}',\mathfrak{s}')\right)\in
  A_q(n;r,s)$ be defined by 
  \[\left((\mathfrak{r},\mathfrak{s})|
    (\mathfrak{r}',\mathfrak{s}')\right):=
  (\mathfrak{r}|\mathfrak{r}')\;
  \mathfrak{det}_q^{(k)}\;(\mathfrak{s}|\mathfrak{s}')^*\]
  whenever $[(\mathfrak{r},\mathfrak{s}),
  (\mathfrak{r}',\mathfrak{s}')]$ is a 
  rational  $(\rho,\sigma)$-bitableau
  such that $\rho\in \Lambda^+(r-k)$, $\sigma\in
  \Lambda^+(s-k)$, for some $k = 0, 1, \dots, \min(r,s)$.
\end{defn}

Note that the proof of Lemma~\ref{lem:factor} and
Remark~\ref{rem:properties}\eqref{item:quantum_determinant_center} show that
$\iota(\mathfrak{det}_q^{(k)})=\mathrm{det}_q^k$. 
Furthermore, if $\rho_1$ or $\sigma_1>n$, then the bideterminant of a
$(\rho,\sigma)$-bitableau vanishes. 
As a direct  consequence of Theorem~\ref{thm:jacobi} we get
\begin{lem}
  Let $(\mathfrak{r},\mathfrak{s})$ and 
  $(\mathfrak{r}',\mathfrak{s}')$ be 
  two standard rational tableaux, and let $\mathfrak{t}$
  and $\mathfrak{t}'$ be the (standard) tableaux obtained 
  from the correspondence of
  Lemma~\ref{lem:bijection}. Then we have
  \[
  \iota((\mathfrak{r},\mathfrak{s})|(\mathfrak{r}',\mathfrak{s}'))
  =(-q)^{c(\mathfrak{t},\mathfrak{t}')}
  (\mathfrak{t}|\mathfrak{t}')
  \]
  for some integer ${c(\mathfrak{t},\mathfrak{t}')}$.
  In particular, the bideterminants of standard rational bitableaux
  are linearly independent.  
\end{lem}
\begin{proof}
This follows directly from Theorem~\ref{thm:jacobi}, the construction
of the bijection and
$\iota(\mathfrak{det}_q^{(k)})=\mathrm{det}_q^k$. The second
statement follows from the fact that the 
$(\mathfrak{t}|\mathfrak{t}')$'s are linearly
independent. 
\end{proof}

\begin{lem}\label{lem:detk}
We have 
\begin{gather}
  \sum_{l=1}^n x_{il}\mathfrak{det}_q^{(k)}x_{jl}^*
  =0\text{ for }i\neq j\label{equ:det1}\\
  \sum_{l=1}^n q^{2l}x_{li}\mathfrak{det}_q^{(k)}x_{lj}^*
    =0\text{ for }i\neq j
  \label{equ:det2}\\
  \sum_{l=1}^n q^{2l-2i}x_{li}\mathfrak{det}_q^{(k)}x_{li}^*=
  \sum_{l=1}^n x_{jl}\mathfrak{det}_q^{(k)}x_{jl}^*\label{equ:det3}
\end{gather}
\end{lem}
\begin{proof}Without loss of generality, we may assume that $k=1$. 
  Suppose  that $i,j\neq 1$. Then
  \begin{align*}
    \sum_{l=1}^n& x_{il}\mathfrak{det}_q^{(1)}x_{jl}^*=
    \sum_{k,l=1}^n x_{ik}x_{1l}x_{1l}^*x_{jk}^*
    =\sum_{k<l} x_{1l}x_{ik}x_{jk}^*x_{1l}^*
    +q^{-2}\sum_{k} x_{1k}x_{ik}x_{jk}^*x_{1k}^*\\
    &+\sum_{k>l} \left(x_{1l}x_{ik}x_{jk}^*x_{1l}^*
      +(q^{-1}-q)(x_{1k}x_{il}x_{1l}^*x_{jk}^*+
      x_{1l}x_{ik}x_{1k}^*x_{jl}^* ) \right)\\
    &=\sum_{k,l} x_{1l}x_{ik}x_{jk}^*x_{1l}^*
    +(q^{-2}-1)\sum_{k} q x_{1k}x_{ik}x_{1k}^*x_{jk}^*\\
    &+(q^{-1}-q)\sum_{k>l}(x_{1k}x_{il}x_{1l}^*x_{jk}^*+
    x_{1l}x_{ik}x_{1k}^*x_{jl}^* )\\
    &=\delta_{ij}\mathfrak{det}_q^{(2)}
    +(q^{-1}-q)\sum_{k,l}x_{1k}x_{il}x_{1l}^*x_{jk}^*
    =\delta_{ij}\mathfrak{det}_q^{(2)} .
  \end{align*}
For $j\neq 1$ we have 
\begin{align*}
    \sum_{l=1}^n& x_{1l}\mathfrak{det}_q^{(1)}x_{jl}^*=
    \sum_{k,l=1}^n x_{1k}x_{1l}x_{1l}^*x_{jk}^*
    =\sum_{k<l} qx_{1l}x_{1k}x_{jk}^*x_{1l}^*
    +q^{-1}\sum_{k} x_{1k}x_{1k}x_{jk}^*x_{1k}^*\\
    &+\sum_{k>l} \left(q^{-1}x_{1l}x_{1k}x_{jk}^*x_{1l}^*
      +(q^{-1}-q) x_{1k}x_{1l}x_{jl}^*x_{1k}^*\right)\\
    &=\sum_{k,l} q^{-1}x_{1l}x_{1k}x_{jk}^*x_{1l}^*=0.
  \end{align*}
Similarly, one can show that
\begin{eqnarray*}
   \sum_{l=1}^n x_{il}\mathfrak{det}_q^{(1)}x_{1l}^*
   &=&0\text{ for }i\neq 1\\
   \sum_{l=1}^n q^{2l-2i}x_{li}\mathfrak{det}_q^{(1)}x_{lj}^*
   &=&\delta_{ij}\sum_{l=1}^n q^{2l-2}x_{l1}
   \mathfrak{det}_q^{(1)}x_{l1}^*
   \text{ for }i,j\neq 1\\
   \sum_{l=1}^n q^{2l-2}x_{l1}\mathfrak{det}_q^{(1)}x_{lj}^*
   &=&0\text{ for }j\neq 1\\ 
   \sum_{l=1}^n q^{2l-2i}x_{li}\mathfrak{det}_q^{(1)}x_{l1}^*
   &=&0\text{ for }i\neq 1.\\
\end{eqnarray*}
Finally, 
\begin{align*}
  \sum_{l=1}^n &q^{2l-2}x_{l1}\mathfrak{det}_q^{(1)}x_{l1}^*
  =\sum_{l,k} q^{2l-2}x_{l1} x_{1k} x_{1k}^*x_{l1}^*
  =\sum_{l,k\neq 1} q^{2l-2} x_{1k}x_{l1}x_{l1}^* x_{1k}^*\\
  &+\sum_{l\neq 1} q^{2l-4} x_{11}x_{l1}x_{l1}^* x_{11}^*
  +\sum_{k\neq 1} q^{2} x_{1k}x_{11}x_{11}^* x_{1k}^*
  + x_{11}x_{11}x_{11}^* x_{11}^*\\
  &=\mathfrak{det}_q^{(2)}
  +\sum_{l\neq 1} q^{2l-4}(1-q^2) x_{11}x_{l1}x_{l1}^* x_{11}^*
  +\sum_{k\neq 1} (q^{2}-1) x_{1k}x_{11}x_{11}^* x_{1k}^*\\
  &=\mathfrak{det}_q^{(2)}+(1-q^2)\left(
    \sum_{l\neq 1} q^{2l-4} x_{11}x_{l1}x_{l1}^* x_{11}^*
    -q^{-2}\sum_{k\neq 1} x_{11}x_{1k}x_{1k}^* x_{11}^*\right)\\
  &=\mathfrak{det}_q^{(2)}.
\end{align*}
The proof is complete.
\end{proof}

\begin{lem}\label{lem:straightening}
  Suppose $\mathbf r=(r_1,\ldots,r_k),\mathbf s=(s_1,\ldots,s_k)\in
  I(n,k)$ are fixed. 
  Let $j\in\{1,\ldots,n\}$ and $k\geq 1$. 
  Then we have, modulo $\mathfrak{det}_q^{(1)}$,
  \begin{eqnarray*}
  &\sum\limits_{j<j_1<j_2<\ldots <j_k}&(\mathbf r|j_k\ldots j_2j_1)_r(\mathbf
  s|j_1j_2\ldots j_k)_r^*\\
  &&\equiv(-1)^kq^{2\sum_{i=0}^{k-1}i}\sum_{j_1<j_2<\ldots <j_k\leq j}
  (\mathbf r|j_k\ldots j_2j_1)_r(\mathbf
  s|j_1j_2\ldots j_k)_r^*
  \end{eqnarray*}
\end{lem}

\begin{proof}
$(\mathbf s|j_1j_2\ldots j_k)_r^*$ and
$(\mathbf s|j_1j_2\ldots j_k)_l^*$ differ only on a power of $-q$
not depending on $j_1,j_2,\ldots,j_k$. Thus we can show the lemma with 
$(\_,\_)_r^*$ replaced by
$(\_,\_)_l^*$. Similarly, we can assume that $r_1<r_2<\ldots <r_k$
and $s_1>s_2>\ldots >s_k$. 
Note that modulo $\mathfrak{det}_q^{(1)}$  we have the relations 
$\sum_{k=1}^n x_{ik}x_{jk}^* \equiv0$.
It follows that the lemma is
true for $k=1$.
Assume that the lemma holds for $k-1$. 
If $M$ is an ordered set, let $M^{k,<}$ be the set of 
$k$-tuples in $M$ with increasing entries. 
For a subset
$M\subset \{1,\ldots,n\}$ we have 
\begin{align*}
  \sum_{\mathbf j\in M^{k,<}}&(\mathbf r|j_k\ldots j_2j_1)_r(\mathbf
  s|j_1j_2\ldots j_k)_l^*\\
  &=
  \sum_{\mathbf j\in M^{k,<},w}(-q)^{-l(w)}
  (\mathbf r|j_k\ldots j_2j_1)_r
  x_{s_1j_{w1}}^*\ldots x_{s_kj_{wk}}^*\\
  &= \sum_{\mathbf j\in M^{k,<},w}
  (\mathbf r|j_{wk}\ldots j_{w1})_r
  x_{s_1j_{w1}}^*\ldots x_{s_kj_{wk}}^*\\
  &= \sum_{\mathbf j\in M^k}
  (\mathbf r|j_{k}\ldots j_{1})_r
  x_{s_1j_{1}}^*\ldots x_{s_kj_{k}}^*
\end{align*}
Applying Laplace's Expansion, we can write
a quantum minor $(\mathbf r|\mathbf j_1\mathbf j_2)_r$ as a linear
combination of products of quantum minors, say 
\[
(\mathbf r|\mathbf j_1\mathbf j_2)_r
=\sum_lc_l(\mathbf r'_l|\mathbf j_1)_r(\mathbf r''_l|\mathbf j_2)_r.
\]
Then with $\epsilon_k:=(-1)^kq^{2\sum_{i=0}^{k-1}i}$, 
$\mathbf j=(j_1,\ldots,j_k)$ and 
$\mathbf j'=(j_1,\ldots,j_{k-1})$, $D=\{j+1\ldots n\}$ and
$C=\{1\ldots  j\}$, we have
 
\begin{align*}
  \sum_{\mathbf j\in D^{k,<}}&(\mathbf r|j_k\ldots j_2j_1)_r(\mathbf
  s|j_1j_2\ldots j_k)_l^*
  =\sum_{\mathbf j\in D^{k}}
  (\mathbf r|j_{k}\ldots j_{1})_r
  x_{s_1j_{1}}^*\ldots x_{s_kj_{k}}^*\\
  &=\sum_{\mathbf j\in D^{k},l}
  c_l(\mathbf r'_l|j_{k})_r(\mathbf r''_l|j_{k-1}\ldots j_{1})_r
  x_{s_1j_{1}}^*\ldots x_{s_{k-1}j_{k-1}}^*x_{s_kj_{k}}^*\\
  &\equiv\epsilon_{k-1}
  \sum_{\genfrac{}{}{0pt}{}{\mathbf j'\in C^{k-1},l}{j_k>j}}
  c_l(\mathbf r'_l|j_{k})_r(\mathbf r''_l|j_{k-1}\ldots j_{1})_r
  x_{s_1j_{1}}^*\ldots x_{s_{k-1}j_{k-1}}^*x_{s_kj_{k}}^*\\
  &=\epsilon_{k-1}
  \sum_{\genfrac{}{}{0pt}{}{\mathbf j'\in C^{k-1}}{j_k>j}}
  (\mathbf r|j_{k}j_{k-1}\ldots j_{1})_r
  x_{s_1j_{1}}^*\ldots x_{s_{k-1}j_{k-1}}^*x_{s_kj_{k}}^*\\
  &=\epsilon_{k-1}
  \sum_{\genfrac{}{}{0pt}{}{\mathbf j'\in C^{k-1}}{j_k>j}}
  (-q)^{k-1}
  (\mathbf r|j_{k-1}\ldots j_{1}j_{k})_r
  x_{s_kj_{k}}^*x_{s_1j_{1}}^*\ldots x_{s_{k-1}j_{k-1}}^*\\
  &=\epsilon_{k-1}
  \sum_{\genfrac{}{}{0pt}{}{\mathbf j'\in C^{k-1},l}{j_k>j}}
  (-q)^{k-1}
  c_l(\mathbf r_l'|j_{k-1}\ldots j_{1})_rx_{\mathbf r_l''j_{k}}
  x_{s_kj_{k}}^*x_{s_1j_{1}}^*\ldots x_{s_{k-1}j_{k-1}}^*\\
  &\equiv-\epsilon_{k-1}
  \sum_{\mathbf j\in C^{k},l}
  (-q)^{k-1}
  c_l(\mathbf r_l'|j_{k-1}\ldots j_{1})_rx_{\mathbf r_l''j_{k}}
  x_{s_kj_{k}}^*x_{s_1j_{1}}^*\ldots x_{s_{k-1}j_{k-1}}^*\\
  &=-\epsilon_{k-1}
  \sum_{\mathbf j\in C^{k}}
  (-q)^{k-1}
  (\mathbf r|j_{k-1}\ldots j_{1}j_{k})_r
  x_{s_kj_{k}}^*x_{s_1j_{1}}^*\ldots x_{s_{k-1}j_{k-1}}^*\\
  &=-\epsilon_{k-1}
  \sum_{\mathbf j\in C^{k,<}}
  (-q)^{k-1}
  (\mathbf r|j_{k}\ldots j_{1})_r
  (s_ks_1\ldots s_{k-1}|j_1\ldots j_{k})_l^*
  \\
  &=-\epsilon_{k-1}
  \sum_{\mathbf j\in C^{k,<}}
  (-q)^{2(k-1)}
  (\mathbf r|j_{k}\ldots j_{1})_r
  (s_1\ldots s_k|j_1\ldots j_{k})_l^*\\
  &=\epsilon_k\sum_{\mathbf j\in C^{k,<}}
  (\mathbf r|j_k\ldots j_2j_1)_r(\mathbf
  s|j_1j_2\ldots j_k)_l^* 
\end{align*}
and the proof is complete.
\end{proof}

\begin{lem}\label{lem:straightening2}
Let $\mathbf r'$ and $\mathbf s'$ be strictly increasing 
multi indices, considered as
tableaux with one row. Let $i$ be the maximal entry appearing 
and suppose that  $i$ is minimal such
that $i$ violates condition~\eqref{equ:condition}.
Let $I$ be the set of entries appearing in both  $\mathbf r'$ and
$\mathbf s'$, then we have $i\in I$. 
Let $L_1=\{k_1,\ldots,k_{l_1}\}$ 
be the set of entries of $\mathbf r'$  not
appearing in $\mathbf s'$, let $L_2=\{k_1',\ldots,k_{l_2}'\}$
be the set of entries of $\mathbf s'$ not
appearing in $\mathbf r'$ and let
 $i_1<i_2<\ldots
<i_k=i$ be the entries of $I$.

Let $D=\{i_1,\ldots,i_k,i_k+1,i_k+2,\ldots,n\}$ and
$C=\{1,\ldots,n\}\backslash (D\cup L_1\cup L_2)$. Furthermore,
for $j_1,\ldots,j_t\in \{1,\ldots,n\}$ let 
\[m(j_1,\ldots,j_t)=|\{(l,c)\in\{1,\ldots,t\}\times
C:j_l<c\}|.\]
Let $\mathbf
k=(k_1,\ldots,k_{l_1}), \mathbf k'=(k_1',\ldots,k_{l_2}')$
and let $\mathbf r$ and $\mathbf s$ be multi indices of the same
length as $\mathbf r'$ resp.~$\mathbf s'$, then
we have  
\[
\sum_{\mathbf j\in D^{k,<}}
q^{2m(\mathbf j)}
(\mathbf r|\mathbf k j_k\ldots j_1)_r
(\mathbf s|j_1\ldots j_k \mathbf k' )_r^*
\equiv0\text{ modulo }\mathfrak{det}_q^{(1)}.
\]
\end{lem}

\begin{proof}
  Note that $i\in I$ and $i=2k+l_1+l_2-1$, otherwise $i-1$ would violate 
  condition~\eqref{equ:condition}. Thus $|C|=k-1$.
  Let $c_{max}$ be the 
  maximal element of $C$, 
  $\tilde{D}=\{c_{max}+1,c_{max}+2,\ldots,n\}
  \subset D\cup L_1\cup L_2$, $\tilde{C}=\{1,\ldots,c_{max}\}$, 
  $D_-=\{d\in D:d<c_{max}\}$ and 
  $D_+=\{d\in D:d>c_{max}\}$. With 
  $\tilde{\mathbf j}=(j_1,\ldots, j_l)$ and 
  $\hat{\mathbf j}=(j_{l+1},\ldots, j_k)$ we have 
  \begin{align}
    \sum_{\mathbf j\in D^{k,<}}&
    q^{2m(\mathbf j)}
    (\mathbf r|\mathbf k j_k\ldots j_1)_r
    (\mathbf s|j_1\ldots j_k \mathbf k' )_r^*\notag\\
    &=\sum_{l=0}^k
    \sum_{\tilde{\mathbf j}\in D_-^{l,<}}
    q^{2m(\tilde{\mathbf j})}
    \sum_{\hat{\mathbf j}\in D_+^{k-l,<}}
    (\mathbf r|\mathbf k j_k\ldots j_1)_r
    (\mathbf s|j_1\ldots j_k \mathbf k' )_r^* .
    \label{equ:xj}
  \end{align}
  Without loss of generality we may assume that the entries in
  $\mathbf s$ are increasing. 
  We apply Laplace's Expansion and 
  Lemma~\ref{lem:straightening} to get 
  for fixed $l$ and $\tilde{\mathbf j}$
  \begin{align*}
    \sum_{\hat{\mathbf j}\in D_+^{k-l,<}}&
    (\mathbf r|\mathbf k j_k\ldots j_1)_r
    (\mathbf s|j_1\ldots j_k \mathbf k' )_r^*
    =\sum_{\hat{\mathbf j}\in \tilde{D}^{k-l,<}}
    (\mathbf r|\mathbf k j_k\ldots j_1)_r
    (\mathbf s|j_1\ldots j_k \mathbf k' )_r^*\\
    &=q^{2l(k-l)}
    \sum_{\hat{\mathbf j}\in \tilde{D}^{k-l,<}}
    (\mathbf r|\mathbf k j_l\ldots j_1j_k\ldots j_{l+1})_r
    (\mathbf s|j_{l+1}\ldots j_k j_1\ldots j_l\mathbf k')_r^*\\
    &\equiv\epsilon_{k-l}q^{2l(k-l)}
    \sum_{\hat{\mathbf j}\in \tilde{C}^{k-l,<}}
    (\mathbf r|\mathbf k j_l\ldots j_1j_k\ldots j_{l+1})_r
    (\mathbf s|j_{l+1}\ldots j_k j_1\ldots j_l\mathbf k')_r^*\\
    &=\epsilon_{k-l}q^{2l(k-l)}\!\!\!\!\!\!\!\!
    \sum_{\hat{\mathbf j}\in (C\cup D_-)^{k-l,<}}
    (\mathbf r|\mathbf k j_l\ldots j_1j_k\ldots j_{l+1})_r
    (\mathbf s|j_{l+1}\ldots j_k j_1\ldots j_l\mathbf k')_r^* .
  \end{align*}
  This expression can be substituted into \eqref{equ:xj}. 
  Each nonzero summand belongs to a disjoint union 
  $S_1\dot\cup S_2=S\subset C\cup D_-$ such that $|S|=k$, 
  $S_1=\{j_1,\ldots,j_l\}$ and $S_2=\{j_{l+1},\ldots,j_k\}$. 
  We will show that the summands belonging to some fixed set $S$
  cancel out. 

  Therefore, we claim that for each subset $S\subset C\cup D_-$ 
  with $k$ elements
  there exists some $d\in D\cap S$ such that 
  $m(d)=|\{s\in S:s>d\}|$.  Suppose not.
  $S$ contains at least one element of $D$
  since $|C|=k-1$. 
  Let $s_1<s_2<\ldots <s_m$ be the elements 
  of $D\cap S$. We show by downward induction that 
  $m(s_l)>|\{s\in S:s>s_l\}|$ for  $1\leq l\leq m$:
  $m(s_m)$ is
  the cardinality of 
  $\{s_m+1,\ldots,c_{max}\}\cap C$.
  Since all $s\in S$ with $s>s_m$ are elements of $C$ we have 
  $\{s_m+1,\ldots,c_{max}\}\cap S\subset 
  \{s_m+1,\ldots,c_{max}\}\cap C$, and thus 
  $m(s_m)\geq|\{s\in S:s>s_m\}|$. By assumption we have
  $>$ instead of $\geq$. 
  Suppose now, that 
  $m(s_l)>|\{s\in S:s>s_l\}|$. 
  We have $\{s\in S:s_{l-1}<s\leq s_l\}=
  \{s\in S\cap C:s_{l-1}<s< s_l\}\cup\{s_l\}$, thus $S$
  contains at most $m(s_{l-1})-m(s_l)$
  elements between  $s_{l-1}$ and $s_l$, and thus at most 
  $m(s_{l-1})-m(s_l)+1+m(s_l)-1
  =m(s_{l-1})$ elements $>s_{l-1}$. By assumption we have 
  $m(s_{l-1})>|\{s\in S:s>s_{l-1}\}|$.
  We have shown that $S$ contains less than $m(s_1)$
  elements greater than $s_1$,
  thus $S$ contains less than  $ |C|+1=k$ elements
  which is a contradiction. This shows the claim. 

  Let $S\subset C\cup D_-$ be fixed subset of cardinality $k$. 
  By the previous consideration there is an element $d\in D\cap S$
  with $m(d)=|\{s\in S:s>d\}|$.
  We claim that the summand
  for $S_1,S_2$ with $d\in S_1$ cancels the summand for   
  $S_1\backslash\{d\},S_2\cup\{d\}$. 
  Note that
  \begin{align*}
    (\mathbf r|&\mathbf k j_l\ldots\hat{d}
    \ldots j_1j_k\ldots{d}\ldots j_{l+1})_r
    (\mathbf s|j_{l+1}\ldots{d}\ldots j_k
    j_1\ldots\hat{d}\ldots j_l\mathbf k')^*_r\\
    &= q^{2|\{s\in S:s>d\}|-2(l-1)}
    (\mathbf r|\mathbf k j_l\ldots j_1j_k\ldots j_{l+1})_r
    (\mathbf s|j_{l+1}\ldots j_kj_1\ldots j_l\mathbf k')^*_r .
  \end{align*}
  Comparing coefficients, we see that both summands cancel. 
\end{proof}

\begin{thm}[Rational Straightening Algorithm]
  The set of bideterminants of standard rational bitableaux forms an
  $R$-basis of $A_q(n;r,s)$.  
\end{thm}
\begin{proof}
  We have to show that the bideterminants of standard rational
  bi\-ta\-bleaux generate $A_q(n;r,s)$.  Clearly, the bideterminants
  $((\mathfrak{r},\mathfrak{s})|(\mathfrak{r}',\mathfrak{s}'))$ with
  $\mathfrak{r},\mathfrak{r}',\mathfrak{s},\mathfrak{s}'$ standard
  tableaux generate $A_q(n;r,s)$.  Let $\mathrm{cont}(\mathfrak{r})$
  resp.~$\mathrm{cont}(\mathfrak{s})$ be the content of $\mathfrak{r}$
  resp.~$\mathfrak{s}$ defined in Definition~\ref{defn:content}.
  
  Let 
  $\mathfrak{r},\mathfrak{r}',\mathfrak{s},\mathfrak{s}'$ be 
  standard tableaux and suppose that
  the rational bitableau
  $[(\mathfrak{r},\mathfrak{s}),(\mathfrak{r}',\mathfrak{s}')]$ is not
  standard. It suffices to show that the bideterminant
  $((\mathfrak{r},\mathfrak{s})|(\mathfrak{r}',\mathfrak{s}'))$ is a
  linear combination of bideterminants
  $((\hat{\mathfrak{r}},\hat{\mathfrak{s}})|
  (\hat{\mathfrak{r}}',\hat{\mathfrak{s}}'))$ such that
  $\hat{\mathfrak{r}}$ has fewer boxes than $\mathfrak{r}$ or
  $\mathrm{cont}(\mathfrak{r})>\mathrm{cont}(\hat{\mathfrak{r}})$ or
  $\mathrm{cont}(\mathfrak{s})>\mathrm{cont}(\hat{\mathfrak{s}})$ in
  the lexicographical order.  Without loss of generality we make the
  following assumptions:
  \begin{itemize}
  \item In the nonstandard rational bi\-ta\-bleau
    $[(\mathfrak{r},\mathfrak{s}),(\mathfrak{r}',\mathfrak{s}')]$ the
    rational tableau $(\mathfrak{r}',\mathfrak{s}')$ is
    nonstandard. Note that the automorphism of
    Remark~\ref{rem:automorphism} maps a bideterminant
    $((\mathfrak{r},\mathfrak{s})|(\mathfrak{r}',\mathfrak{s}'))$ to
    the bideterminant
    $((\mathfrak{r}',\mathfrak{s}')|(\mathfrak{r},\mathfrak{s}))$.
  \item Suppose that $(\mathfrak{r},\mathfrak{s})$ 
    and $(\mathfrak{r}',\mathfrak{s}')$ are
    $(\rho,\sigma)$-tableaux. In view of Lemma~\ref{lem:detk} we can
    assume that  $\rho\in \Lambda^+(r)$ and
    $\sigma\in\Lambda^+(s)$. 
  \item $\mathfrak{r},\mathfrak{r}',\mathfrak{s},\mathfrak{s}'$
    are tableaux with only one row (each bideterminant 
    has a factor of this type, and we can use
    Theorem~\ref{thm:standardbasis} to write nonstandard
    bideterminants as a linear combination of standard ones of the
    same content. 
  \item Let $i$ be minimal such that condition~\eqref{equ:condition} of 
    Definition~\ref{defn:rationaltableaux}   is violated for $i$. 
    Applying Laplace's Expansion, we may assume that there is no greater
    entry than $i$ in $\mathfrak{r}'$ and in $\mathfrak{s}'$.
  \end{itemize}
  Note that all elements of $A_q(n;r,s)$
  having a factor $\mathfrak{det}_q^{(1)}$
  can be written as a linear combination of
  bideterminants of rational $(\rho,\sigma)$-bitableaux with
  $\rho\in\Lambda^+(r-k), k>0$. Thus, it suffices to show that 
  $((\mathfrak{r},\mathfrak{s})|(\mathfrak{r}',\mathfrak{s'}))$ is, 
  modulo $\mathfrak{det}_q^{(1)}$,  a linear
  combination of bideterminants 
  of 'lower content'. 
  The  summand of highest content 
  in Lemma~\ref{lem:straightening2} is that one
  for $\mathbf j=(i_1,i_2,\ldots,i_k)$, and this summand 
  is a scalar multiple (a power of $-q$, which is invertible)
  of $((\mathfrak{r},\mathfrak{s})|(\mathfrak{r}',\mathfrak{s}'))$. 
\end{proof}

The following is an immediate consequence of the preceding theorem and
Lemma~\ref{lem:bijection}.

\begin{cor}
There exists an $R$-linear map $\phi:A_q(n,r+(n-1)s)\to A_q(n;r,s)$
given on a basis by $\phi(\mathfrak{t}|\mathfrak{t}')
:=(-q)^{-c(\mathfrak{t},\mathfrak{t}')}
((\mathfrak{r},\mathfrak{s})|(\mathfrak{r}',\mathfrak{s}'))$ if the
shape $\lambda$ of $\mathfrak{t}$ satisfies $\sum_{i=1}^s\lambda_i\geq
(n-1)s$  where $(\mathfrak{r},\mathfrak{s})$ and
$(\mathfrak{r}',\mathfrak{s}')$ are the rational tableaux respectively
corresponding to $\mathfrak{t}$ and $\mathfrak{t}'$ under the
correspondence of Lemma~\ref{lem:bijection}, and
$\phi(\mathfrak{t}|\mathfrak{t}') := 0$ otherwise. We have
\[\phi\circ\iota=\mathrm{id}_{A_q(n;r,s)}\]
and thus $\pi=\iota^*$ is surjective.
\end{cor}

As noted in Section~\ref{section:submodule} we now have the main
result.

\begin{thm}[Schur--Weyl duality for mixed tensor space, II]
  \label{thm:schurweyl2}
  \begin{eqnarray*}
    S_q(n;r,s)&=&\mathrm{End}_{\mathfrak{B}_{r,s}(q)}
    (V^{\otimes r}\otimes {V^*}^{\otimes s})
    =\rho_{\mathrm{mxd}}({\mathbf U})=\rho_{\mathrm{mxd}}({\mathbf U}')
  \end{eqnarray*}
Furthermore, 
$S_q(n;r,s)$ is $R$-free with a basis 
indexed by standard rational bitableau. 
\end{thm}
\begin{proof}
  The first assertion follows from the surjectivity of $\pi$, the
  second assertion is obtained by dualizing the basis of $A_q(n;r,s)$. 
\end{proof}



\begin{thebibliography}{00}
\frenchspacing\small

\bibitem{bchlls}
  G. Benkart, M. Chakrabarti, T. Halverson, R. Leduc, C. Lee, J. Stroomer,
  Tensor product representations of 
  general linear groups and their connections 
  with Brauer algebras, J. Alg. 166 (1994) 529--567.

\bibitem{birmanwenzl}
  J.  Birman, H. Wenzl,
  Braids, link polynomials and a new algebra, 
  Trans. Amer. Math. Soc. 313 (1989) 249--273.

\bibitem{brauer}
  R. Brauer,
  On algebras which are connected 
  with the semisimple continuous groups, 
  Ann. of Math. 38 (1937) 857--872.

\bibitem{BrundanStroppel} J.~Brundan and C.~Stroppel, Gradings on
  walled Brauer algebras and Khovanov's arc algebras, preprint,
  arXiv:1107.0999.

\bibitem{dipperdonkin}
  R. Dipper, S. Donkin,
  Quantum $GL_n$, 
  Proc. London Math. Soc. (3) 63 (1991) 165--211.
	
\bibitem{dipperdoty}
  R. Dipper, S. Doty,
  The  rational Schur algebra, 
  Represent. Theory  12  (2008) 58--82.
	
\bibitem{dipperdotystoll}
  R. Dipper, S. Doty, F. Stoll, 
  The quantized walled Brauer algebra and mixed tensor space, 
  to appear, Algebras and Representation Theory,
  preprint available on arXiv:0806.0264.
	



\bibitem{dipperjames}
  R. Dipper, G. James,
  Representations of Hecke algebras 
  of general linear groups, 
  Proc. London Math. Soc. (3) 52 (1986) 20--52.

\bibitem{dipperjamesschur}
  R. Dipper, G. James,
  The $q$-Schur algebra, 
  Proc. London Math. Soc. 59 (1989) 23--50.




\bibitem{goodearl}
  K. R. Goodearl,
  Commutation Relations for Arbitrary Quantum Minors, 
  Pacific J. Math. 228 (1) (2006)  63--102.



 \bibitem{green}
   R. M. Green,
   $q$-Schur algebras as quotients of quantized enveloping algebras,
   J. Alg. 185 (1996)  660--687.


\bibitem{hongkang}
  J. Hong, S. Kang,
  Introduction to quantum groups and crystal bases, 
  American Math. Soc. 2002.

\bibitem{huangzhang}
   R.~Q.~Huang,   J.~J.~Zhang,
   Standard Basis Theorem for Quantum Linear Groups,
   Adv.~Math. 102 (1993) 202--229.


\bibitem{jantzen}
  J. C. Jantzen,
  Lectures on quantum groups, 
  American Math. Soc. 1995.

\bibitem{Koike} K.~Koike, On the decomposition of tensor products of
  the representations of the classical groups: by means of the
  universal characters, Adv. Math. 74 (1989), 5--86.

 
\bibitem{kosuda}
   M. Kosuda,
   Representation of $q$-analogue of rational Brauer algebras,
   Tsukuba J. Math. 21  (1997) 707--728.

\bibitem{kosudamurakami}
  M. Kosuda, J. Murakami,
  Centralizer algebras of the mixed tensor representations of
  the quantum algebra $U_q(gl(n,\mathbb{C}))$, 
  Osaka J. Math. 30 (1993) 475--507.


\bibitem{leduc}
  R. Leduc,
  A two-parameter version of the centralizer algebra of the mixed
  tensor representation of the general linear group and quantum general
  linear group,
  thesis, University of Wisconsin--Madison, 1994.
  

\bibitem{lusztig}
  G. Lusztig,
  Finite dimensional Hopf algebras arising
  from quantized universal enveloping algebras,   
  J. Amer. Math. Soc. 3 (1990) 257--296.

\bibitem{murakami}
   J. Murakami,
  The Kauffman polynomial of links and representation theory,
   Osaka J. Math. 24 (1987)  745--758.

 \bibitem{schur}
   I. Schur, 
   "Uber die rationalen Darstellungen der allgemeinen linearen Gruppe,
   (1927) reprinted in: I. Schur, Gesammelte Abhandlungen III, Springer,
   Berlin (1973) 68--85.
   
 \bibitem{stembridge}
   J.~R.~Stembridge,
   Rational Tableaux and the Tensor Algebra of $\mathrm{gl}_n$,
   J.~Comb.~Theory, Ser.~A  46 (1987) 79--120.

\bibitem{tange} 
  R. Tange,
  A bideterminant basis for a reductive monoid,
   J. Pure Appl. Algebra 216 (2012), 1207--1221.

\bibitem{Turaev} V.~G.~Turaev, Operator invariants of tangles and
  $R$-matrices (Russian) Izv. Akad. Nauk SSSR Ser. Mat. 53 (1989),
  1073--1107, 1135; translation in Math. USSR-Izv. 35 (1990),
   411--444.



\end{thebibliography}
\end{document}